\documentclass[a4paper]{amsart}
\usepackage{float,subcaption}
\usepackage[T1]{fontenc}
\usepackage[utf8]{inputenc}
\usepackage{mathtools}
\usepackage{thmtools}
\usepackage{thm-restate}
\usepackage{bm}
\usepackage{stmaryrd}
\usepackage{amssymb}
\usepackage{todonotes}

\usepackage[inline]{enumitem}
\usepackage{soul}
\usepackage[normalem]{ulem}

\usepackage{tikz}

\usepackage[dvipsnames]{xcolor}
\usepackage{hyperref}
\definecolor{linkblue}{named}{MidnightBlue}
\hypersetup{colorlinks=true, linkcolor=linkblue,  anchorcolor=linkblue,
	citecolor=linkblue, filecolor=linkblue, menucolor=linkblue,
	urlcolor=linkblue}

\usetikzlibrary{positioning,arrows,calc,matrix,tqft,through,quotes}
\hypersetup{
    colorlinks,
    linkcolor={blue!50!black},
    citecolor={blue!50!black},
    urlcolor={blue!80!black}
}

\usepackage{pat}

\newcommand{\sm}{\smallsetminus}

\usetikzlibrary{positioning,arrows,calc,matrix,tqft,through}\usepackage[T1]{fontenc}
\usetikzlibrary{calc,math,lindenmayersystems}

\newcommand{\nocontentsline}[3]{}
\let\origcontentsline\addcontentsline
\newcommand\stoptoc{\let\addcontentsline\nocontentsline}
\newcommand\resumetoc{\let\addcontentsline\origcontentsline}

\crefname{p}{}{}
\creflabelformat{p}{#2(#1)#3}

\usepackage{todonotes}

\usepackage[longnamesfirst,numbers,sort&compress]{natbib}

\newcommand{\Ucal}{\ensuremath{\mathcal{U}}}
\newcommand{\Vcal}{\ensuremath{\mathcal{V}}}
\newcommand{\Wcal}{\ensuremath{\mathcal{W}}}

\newcommand{\cay}{\mathsf{Cay}}

\newcommand{\diam}{\mathsf{diam}}
\newcommand{\ab}{\ensuremath{(A,B)}}
\newcommand{\ba}{\ensuremath{(B,A)}}
\newcommand{\dc}{\ensuremath{(D,C)}}
\newcommand{\cd}{\ensuremath{(C,D)}}

\newcommand{\cartprod}{\mathbin{\Box}}

\newenvironment{txteq*}
{
	\begin{equation*}
	\begin{minipage}[t]{0.85\textwidth} 
	\em                                
}
{\end{minipage}\end{equation*}\ignorespacesafterend}
\newenvironment{claimproof}{%
  \par\noindent\textit{Proof of Claim.} \ignorespaces
}{%
  \hfill$\lozenge $\par
}
\definecolor{OliveGreen}{RGB}{120, 140, 47}

\title{Vertex cuts and median decompositions}
\author{Joseph P. MacManus and Bobby Miraftab}

\date{17 June, 2026}

\begin{document}

\begin{abstract}
Median decompositions were introduced by Stavropoulos in 2015 as a generalisation of tree decompositions. In this paper, we further develop and exposit this theory as a tool in structural graph theory to study systems of vertex separations.

Generalising the well-known fact that nested systems of vertex separations produce tree decompositions of a graph over the structure tree, we describe how a (not necessarily nested) system of separations produces a median decomposition. The median graph in this decomposition is the `dual median graph' constructed by Sageev. If the system of cuts is nested then this median decomposition recovers precisely the aforementioned tree decomposition. We prove a theorem asserting that this decomposition is `uniquely minimal', and describe how Sageev--Roller duality manifests in median decompositions.

As an application of our structural approach, we extend a theorem of Stavropoulos from finite graphs to all graphs, which states that the median-width a graph is equal to its clique number. 

We also describe the link between (canonical) median decompositions and (equivariant) coarse embeddings/quasi-isometries into median graphs. A corollary of these results is a characterisation of when a finitely generated group acts metrically-properly/geometrically on a median graph, in terms of canonical median decompositions of its Cayley graphs.

\end{abstract}

\maketitle


\section{Introduction}

It is a fundamental fact, originating in work of Dunwoody, that a nested system of vertex separations $\Sigma$ of a graph determines a tree $T_\Sigma$, sometimes called the \emph{structure tree}. This viewpoint underlies the theory of \emph{tree decompositions}. A tree decomposition may be seen as a way of encoding a nested family of vertex cuts in a form that exposes both the global structure of the graph and the local connectivity within each part.
Tree decompositions have become an indispensable tool in structural graph theory. Introduced first by Halin in \cite{halin}, they played a key role in the famous Graph Minors project of Robertson and Seymour. The construction of a tree decomposition over the structure tree was described by Carmesin et al. in the finite case \cite{connectivity}, and Elbracht et al. in the infinite case \cite{elbracht2022trees}. The edges of a tree decomposition define a system of separations, and these operations are, in a sense, dual to each other.


The effectiveness of the structure tree and tree decompositions rests heavily on the assumption that the underlying system is nested. Of course, many naturally arising separation systems in graph theory do not have this property, featuring pairs of separations which \emph{cross}. One way to overcome this limitation is given by a classical construction of Sageev \cite{sageev1995ends} (see also \cite{nica2004cubulating,chatterji2005wall}). This associates to a non-nested system of separations a higher-dimensional object encoding its nesting and crossing structure. This object is a \emph{median graph}---or equivalently, in more geometric terms, a CAT(0) cube complex \cite{roller-thesis, chepoi-cat0}. In this paradigm, it is the \emph{$\Theta$-classes}, or \emph{hyperplanes}, which encode the separations, rather than the edges. When the input system of separations is totally nested, Sageev's construction recovers the structure tree exactly. 
Since Sageev's seminal work, the study of groups acting on median graphs and CAT(0) cube complexes has grown into a hugely powerful toolbox within geometric group theory with countless applications. See \cite{schwer2023cat, hruska2014finiteness} for introductions to and surveys of this topic.

The passage from trees to median graphs suggests a natural generalisation of tree decompositions, where we instead organise over a median graph. Such objects, called \emph{median decompositions}, were introduced by Stavropoulos in 2015 \cite{stavropoulos2015medianwidth}. Median decompositions have so far been used to study the cops-and-robber game \cite{stavropoulos2016cops} as well as the $k$-chordality problem \cite{MR4968370}.

The goal of the present paper is to advertise and exposit median decompositions as a useful tool in structural graph theory, further developing their theory and bringing it more in-line with the rich study of cubulation in geometric group theory. We will pay particular attention to their link with Sageev's construction, and how the phenomenon of Sageev--Roller duality presents itself within median decompositions. While this theory is closely related to the usual `spaces with walls' paradigm, we remark that it is not quite the same. The key distinction is that the separations we consider in our graphs are \emph{vertex separations}, meaning the two half-spaces overlap (a similar set-up is used by Hsu--Wise in \cite{hsu2010cubulating} in the context of 2-complexes, though not in this generality). This, in-turn, equips the dual median graph with additional data encoding how these separations interact locally. Morally, it is precisely this additional data which leads to the notion of a median decomposition, which is perhaps better-suited to applications in structural graph theory. A core idea we hope to put across is that this additional data is also `seen' by Sageev--Roller duality.

Firstly, parallel to the tree decompositions of Carmesin et al. and Elbracht et al. constructed over the structure tree of a nested system, we describe how to construct a median decomposition over the dual median graph of a (not necessarily nested) system of separations.

\begin{restatable}[cf.~\Cref{cor:uniqueness}]{thm}{main-result-intro}\label{thm:main-result-intro}
Let $G$ be a graph, and let $\Sigma$ be a discrete, crossing-finite, ECC system of separations. Then there exists a \textbf{unique} reduced, crossing-faithful median decomposition $(M, \beta)$, such that the system of separations $\Sigma_{M,\beta}$ of $G$ induced by the $\Theta$-classes of $M$ is exactly $\Sigma$. 
\end{restatable}

The median graph $M$ in \Cref{thm:main-result-intro} is the dual median graph produced by Sageev's construction, and we refer to the median decomposition $(M,\beta)$ as the \emph{dual median decomposition}. See \Cref{thm:median_decomposition} for details of its construction. We remark that if the system $\Sigma$ in \Cref{thm:main-result-intro} is nested, then the construction of the dual median decomposition precisely recovers the usual tree decomposition over the structure tree. This theorem can be interpreted as a manifestation of Sageev--Roller duality for median decompositions; every (sufficiently sensible) system of separations is associated to a `good' median decomposition, every such median decomposition induces a system of separations, and these operations are inverse to one another. 

 The crossing-faithful hypothesis (see \Cref{def:crossing-faithful}) in the above is necessary for the uniqueness statement. However, median decompositions which do not satisfy this assumption are still closely related to the dual median decomposition by the following elaboration on \Cref{thm:main-result-intro}. 

 \begin{thm}[cf.~\Cref{prop:median-decomp-contains-the-dual}]
     Let $G$ be a graph, and let $\Sigma$ be a discrete, crossing-finite, ECC system of separations. Let $(M, \beta)$ be the dual median decomposition. Let $(N, \alpha)$ be another reduced median decomposition such that the system of separations $\Sigma_{N,\alpha}$ of $G$ induced by the $\Theta$-classes of $N$ is exactly $\Sigma$. Then there exists a median embedding $\varphi : M \hookrightarrow N$ such that $\alpha \circ \varphi = \beta$. The map $\varphi$ is surjective if and only if $(M',\alpha)$ is crossing-faithful.
 \end{thm}

 In other words, the dual median decomposition is \emph{uniquely minimal} amongst median decompositions encoding a given system of separations, in a certain sense.

Much like how tree decompositions allow the definition of the parameters such as tree-width, the same definitions make sense for median decompositions.
 In \cite{stavropoulos2015medianwidth}, Stavropoulos showed that the median width of a finite graph is exactly equal to the size of its largest clique. We show how our structural approach allows us to strengthen this result from finite graphs to all graphs.

\begin{restatable}[{cf.~\Cref{cor:clique-med-decomp}}]{thm}{clique}\label{thm:main-result-intro-2}
 Let $G$ be a graph. Then $\operatorname{mw}(G) = \omega(G)$. 
\end{restatable}

Our proof of \Cref{thm:main-result-intro-2} is quite different in spirit to the proof of the finite case in \cite{stavropoulos2015medianwidth}, making use of the dual median decomposition mentioned above. In contrast, Stavropoulos' main tool for constructing median decompositions involves considering a `Cartesian product of tree decompositions'. We briefly recount this method in \Cref{sec:clique}.

Regarding \Cref{thm:main-result-intro-2}, an important lesson is in order. It shows that median-width alone does not encode any about the large-scale geometry of a graph. This is in stark contrast with tree-width, which is known to strongly influence the global coarse structure of a graph \cite{hickingbotham2025graphs, nguyen2025coarse, distel2025alternative}.
This is all, of course, not to say that median decompositions do not influence the global geometry of graphs. When additional hypotheses are in place, median decompositions can strongly describe the global geometry of a graph. This is explored in \Cref{sec:geometry} of this paper through the lens of what we call \emph{proper} and \emph{geometric} median decompositions. We will abstain from defining these terms in this introduction, but remark that `proper' decompositions correspond sharply to \emph{coarse embeddings} into median graphs, while `geometric' decompositions characterise quasi-isometries. 
Some of the results in this section are summarised by the following, which can be read as generalising \cite[Thm.~1.4]{hruska2014finiteness}. 

\begin{thm}[cf.~\Cref{cor:coarse-embed-trichotomy}]\label{thm:intro-geom-trichotomy}
    Let $G$ be a graph equipped with an action by a group $\Gamma$. Then the following are equivalent. 
    \begin{enumerate}
        \item\label{thm:trichotomy1} There exists a $\Gamma$-equivariant coarse embedding of $G$ into a crossing-bounded median graph. 

        \item\label{thm:trichotomy2} There exists a $\Gamma$-canonical, reduced, proper median decomposition of $G$ over a crossing-bounded median graph.

        \item\label{thm:trichotomy3} There exists  a $\Gamma$-invariant, discrete, crossing-bounded, ECC separation system of $G$ with uniform local width and increasing separations.
    \end{enumerate}
\end{thm}

A similar result for quasi-isometries---though without the third equivalent condition; see Remark~\ref{rem:quasi-dense} for discussion---is given in Theorem~\ref{thm:proper-geom-char}. 
If no group action is present, one can simply take $\Gamma = \{1\}$ in the above, which is interesting in its own right. For example, it is a theorem of Haglund--Wise that every hyperbolic group is quasi-isometric to a bounded-degree median graph \cite{haglund2012combination}. The results of Section~\ref{sec:geometry} are also closely related to a similar result for \emph{graph decompositions}, which appears in Knappe's thesis; see \cite[\S4.3]{knappe2025graph}.
It is interesting to note that a graph admits a coarse embedding into some crossing-bounded (a.k.a. finite dimensional) median graph if and only if it has finite asymptotic dimension \cite{kasprowski2022coarse, wright2012finite}. In particular, \Cref{thm:intro-geom-trichotomy} suggests a method for bounding the asymptotic dimension of a graph, by exhibiting a sufficiently sensible system of separations.

Another corollary of the results in \Cref{sec:geometry} is the following structural characterisation of when a finitely generated group $\Gamma$ acts metrically-properly (and coboundedly) on a median graph. 

\begin{restatable}[cf.~\Cref{cor:geom-med-decomp-cayley-graphs}]{thm}{mainn}\label{thm:cayley-char}
Let $\Gamma$ be a finitely generated group and $S \subset \Gamma$ a finite generating set. Let $G = \cay(\Gamma,S)$ be the associated Cayley graph. Let $M$ be a median graph. Then $\Gamma$ acts metrically-properly (and coboundedly) on $M$ if and only if $G$ admits a $\Gamma$-canonical proper (resp. geometric) median decomposition over $M$. 
\end{restatable}

This could be compared with a theorem of Kuske--Lohrey, which asserts that a finitely generated group is virtually free (i.e. acts geometrically on a tree) if and only if its Cayley graphs have finite tree-width \cite{kuske2005logical}. These theorems are not direct analogues of each other, however, and an important distinction should be made here. In particular, the median decompositions in \Cref{thm:cayley-char} are required to be canonical, whereas the tree decompositions in \cite{kuske2005logical} have no such assumptions. It is not possible to drop the canonicity from our theorem, as the class of groups acting on metrically properly and coboundedly on median graphs is not closed under quasi-isometries \cite{fournier2023no}, unlike the class of virtually free groups. One striking example of this comes from the mapping class group of a (non-degenerate) finite type surface; such a group is quasi-isometric to a finite-dimensional median graph, but there does not exist another median graph in the same quasi-isometry class which admits a metrically-proper, cobounded action \cite[\S4]{petyt2022large}. In particular, it is impossible to drop the canonicity from Theorem~\ref{thm:cayley-char} while preserving the conclusion. If we \emph{do} drop canonicity, we instead obtain a characterisation of when $\Gamma$ admits a coarse embedding/quasi-isometry to the median graph $M$ in terms of median decompositions; see \Cref{thm:proper-geom-char} for details.

Instead, we suggest that \Cref{thm:cayley-char} perhaps be read as a translation, turning the group-theoretic idea of an action on a median graph to a purely structural statement. Indeed, the hypothesis that a graph `admits a canonical proper/geometric median decomposition' makes sense on any graph with equipped with an action, not just Cayley graphs. In particular, median decompositions give a meaningful way to describe a, say, vertex-transitive graph $G$ as being `cubulated' in the GGT sense, even if $G$ is far from being a Cayley graph.

\subsection*{Outline of this paper}

This paper is structured as follows:

\begin{itemize}
    \item In \Cref{sec:prelims}, we establish some basic terminology surrounding graphs and separation systems.

    \item Next, \Cref{sec:median-graphs} introduces median graphs and their basic properties.

    \item \Cref{sec:med-decomp} discusses median decompositions, with particular attention paid to how a median decomposition of a graph gives rise to a separation system.

    \item In \Cref{sec:sageev}, we recount the details of Sageev's construction of the dual median graph. 

    \item The dual median decomposition is constructed in \Cref{sec:dual-meddecomp}, and its key properties studied. Its uniqueness is shown in \Cref{sec:duality}.

    \item \Cref{sec:clique} contains a short proof of Stavropoulos' characterisation of clique number, applicable to all graphs.

    \item Finally, in \Cref{sec:geometry} we introduce proper and geometric median decompositions, and describe how they correspond to coarse embeddings and quasi-isometries. 
\end{itemize}

\subsection*{Acknowledgments}

The first author was supported by the Additional Funding Programme for Mathematical Sciences, delivered by EPSRC (EP/V521917/1) and the Heilbronn Institute for Mathematical Research. The second author was supported by the NSERC. This research was conducted while the second author was visiting the University of Oxford. 
We are grateful to Mark Hagen for comments.


\section{Preliminaries}\label{sec:prelims}

Before we begin, we quickly establish the notational and terminological conventions of this paper.

Throughout this paper, all graphs considered are simple and undirected. 
Given a graph $G$, we denote by $V(G)$ its vertex set, and $E(G)$ is its set of (unoriented) edges. If $e$ is an edge with endpoints $x$, $y$, we denote this by $e = xy = yx$. Every such edge has two choices of \defin{orientation}, which we denote by the ordered pairs $(x,y)$ and $(y,x)$. The set of oriented edges of $G$ is denoted $\vec E(G)$. Note that there is a natural 2-1 map $\vec E(G) \to E(G)$ given by $(x,y) \mapsto xy$.

An \defin{edge cut}, or just \defin{cut} for short, is an ordered bipartition of the vertex set of $G$ into disjoint non-empty sets $A$ and $B$.
We call the sets $A$ and $B$ the \defin{sides} of the cut and denote this cut by $E[A,B]$. 
A cut \defin{separates} two non-empty subsets $C, D \subset V(G)$ if $C$ and $D$ are contained in distinct sides of the cut. The cut $E[A,B]$ is called \emph{tight} if both $G[A]$ and $G[B]$ are connected. 

A \defin{vertex separation}, or just \defin{separation} for short, is an ordered pair $(A,B)$ of subsets of $V(G)$ such that  $G[A]\sqcup G[B]=G$. In other words,
we have that $A\cup B = V(G)$ and no edge joins $A\sm B$ to $B\sm A$.
A separation $(A,B)$ is said to be \defin{essential}\footnote{These are often called \emph{proper separations} in the literature. We have chosen to use different terminology to avoid overloading the word proper.} if $A \neq V(G)$ and $B \neq V(G)$.
The \defin{flip} of $(A,B)$ is $(B,A)$; we write $(A,B)^{r}\coloneqq(B,A)$.  
For a set $\Sigma$ of separations, let  
$X^{r}\;\coloneqq\;\{(B,A) : (A,B)\in \Sigma\}$
be the set of their reverses.  A set $\Sigma$ of separations is called \defin{symmetric} if $\Sigma=\Sigma^r$. 
A symmetric set of essential separations is called a \defin{system of separations}, or \defin{separation system}. 
The \defin{order} of a separation $\ab$ is $|A\cap B|$.
Every set of separations is partially ordered by  
\[
  (A,B)\;\le\;(C,D) \quad\Longleftrightarrow\quad A\subseteq C\ \text{ and }\ D\subseteq B.
\]
Two separations $(A,B)$ and $(C,D)$ are \defin{nested} if $(A,B)$ is comparable with either $(C,D)$ or $(D,C)$ under this order.
A set of separations $\Sigma$ is \defin{nested} if every pair of elements of $\Sigma$ are nested. 
Two separations  are said to \defin{cross} if they are not nested.
%
%
We say that a separation system $\Sigma$ is \defin{discrete} if for all $(A_1,B_1), (A_2,B_2) \in \Sigma$, we have that the set
$$
\{(C,D) \in \Sigma : (A_1,B_1) \leq (C,D) \leq (A_2, B_2)\}
$$
is finite.
We say that $\Sigma$ is \defin{crossing-finite} (resp. \defin{crossing-bounded}) if any set of pairwise-crossing elements of $\Sigma$ is finite (and uniformly bounded).
If $\Gamma$ is a group acting on $G$, we may refer to $G$ as a \defin{$\Gamma$-graph} for brevity. We will say that a separation system $\Sigma$ is \defin{$\Gamma$-invariant} if for all $(A,B) \in \Sigma$, we have that $(g\cdot A, g \cdot B) \in \Sigma$.

We also introduce the following additional property.

\begin{defn}[Escaping chain condition]
    Let $G$ be a graph and let $\Sigma$ be a system of separations. We say that $\Sigma$ satisfies the \defin{escaping chain condition (ECC)}, and call $\Sigma$ an \defin{ECC separation system}, if for all infinite, strictly descending chains 
    $$
    (A_1, B_1) > (A_2, B_2) > \ldots,
    $$
    we have that 
    $$
    \bigcap_{n=1}^\infty (A_n \sm B_n) = \emptyset.
    $$
\end{defn}

Intuitively, this property forbids any descending chain from `trapping' a vertex. Note that discrete separation systems do not automatically satisfy the ECC, even if they are nested. 

\begin{exa}\label{exa:no-DCC-principal-filter}
    Consider separation system $\Sigma$ on the infinite star-graph $G$ with vertex set $V(G) = \{\omega, b, 1, 2 , 3, \ldots\}$, where $b$ is the base-vertex and every other vertex is a leaf. Consider the separation system $\Sigma = \{(A_n, B_n), \  (B_n, A_n) : n \geq 1\}$ where
    $$
    A_n = \{b, 1, \ldots, n\}, \ \ \ B_n = \{\omega, b, n+1, n+2, \ldots\}. 
    $$
    This is a discrete and nested separation system, but does not satisfy the ECC. 
\end{exa}


\section{Median graphs}\label{sec:median-graphs}

In this section, we introduce the basic conventions and facts surrounding median graphs.

\subsection{Median graphs and convexity}

Median graphs are generalisations of trees, being graphs with the property that any triple of vertices admits a unique `median'. This is made precise as follows.
Let $G$ be a (finite or infinite) graph.
A \defin{shortest path} or \defin{geodesic} between two vertices is a path whose length is minimal among all $(u,v)$‑paths in~$G$.
For vertices $u,v\in V(G)$ we write $d(u,v)$ for this length and call any shortest $(u,v)$‑path a \defin{$(u,v)$‑geodesic}.
The \defin{interval} between $u$ and $v$ is the set
\[
   I(u,v)\;\coloneqq\;\bigl\{x\in V(G) :  d(u,v)=d(u,x)+d(x,v)\bigr\},
\]
i.e.\ the vertices that lie on some $(u,v)$‑geodesic.  

\begin{defn}[Median graphs]
    A graph $G$ is a \defin{median graph} if, for every three distinct vertices
$u,v,w$ of a component of $G$, the triple intersection  
\[
   I(u,v,w)\;\coloneqq\;I(u,v)\cap I(v,w)\cap I(w,u)
\]
contains exactly one vertex.  
That unique vertex is called the \defin{median} of $u,v,w$.
\end{defn}

It is implicit within this definition that median graphs are connected. 
Examples include trees and cube skeleta. It is a theorem of Chepoi \cite{chepoi-cat0} and Roller \cite{roller-thesis} that every median graph is the 1-skeleton of some CAT(0) cube complex, and vice versa. In other words, all median graphs are obtained by gluing together cube graphs, following certain rules.

\begin{defn}
    Let $G$ be a graph, and $C \subset G$ a subgraph. We say that $C$ is \defin{convex} if every geodesic in $G$ with endpoints in $C$ is contained in $C$.
\end{defn}

In trees, convexity is identical to connectedness. Median graphs also exhibit a very strong convexity theory. For example, convex subgraphs of median graphs are easily seen to also be median graphs. Moreover, convex subgraphs of median graphs satisfy a form of the \emph{Helly property}. See \cite[Thm.~12.19]{hammack2011handbook} for a proof.

\begin{thm}[Helly property for median graphs]\label{thm:helly}
    Let $M$ be a median graph, and let $C_1, \ldots , C_n$ be a finite collection of convex subgraphs. If $C_i \cap C_j \neq \emptyset$ for all $i, j \in [n]$, then 
    $$
    \bigcap_{i=1}^n C_i \neq \emptyset.
    $$
\end{thm}

\subsection{\texorpdfstring{$\Theta$}{Theta}-classes}

Next, we introduce one of the fundamental tools used in the study of median graphs.
The Djoković--Winkler $\Theta$-relation (often just called the $\Theta$-relation) is a binary relation on the edge set of a graph that captures when two edges are `parallel'.
We define both oriented and unoriented variants of this relation, starting with the oriented.

\begin{defn}[Oriented $\Theta$-classes]
Let $M$ be a median graph. 
For oriented edges $e= (x,y), f = (u,v) \in \vec E(M)$, define the relation
\[
e\;\Theta f
\quad\Longleftrightarrow\quad
d(x,u) + d(y,v) < d(x,v) + d(y,u).
\]
The equivalence classes of the $\Theta$-relation on $\vec E(M)$ are called \defin{oriented $\Theta$-classes}. Denote by $\vec {\mathcal H}(M)$ the set of oriented $\Theta$-classes in $M$. 
\end{defn}

For the sake of self-containment, we include a quick proof that $\Theta$ is indeed an equivalence relation. 

\begin{prop}
    Let $M$ be a median graph.  
    Then $\Theta$ is an equivalence relation on $\vec E(M)$.
\end{prop}

\begin{proof}
We fix an oriented edge $e=(x,y)$ and define
$\sigma_e:V(M)\to \mathbb Z, \sigma_e(z):=d(z,y)-d(z,x)$.
Since $M$ is bipartite and $xy\in \vec E(M)$, for every $z$ we have $|\sigma_e(z)|=1$.
\begin{clm}
For $e=(x,y)$ and $f=(u,v)$ in $\vec E(M)$,
$e\Theta f\Longleftrightarrow\sigma_e=\sigma_f$.
\end{clm}
\begin{claimproof}
One can see that 
\[
\bigl(d(x,v)+d(y,u)\bigr)-\bigl(d(x,u)+d(y,v)\bigr)
=\bigl(d(x,v)-d(x,u)\bigr)+\bigl(d(y,u)-d(y,v)\bigr).
\]  
Because $uv\in \vec{E}(M)$, each bracket is $\pm 1$, so
\begin{align*}\tag{$\ast$}
e\Theta f
\; &\Longleftrightarrow\;
d(x,v)=d(x,u)+1 \ \text{ and }\ d(y,u)=d(y,v)+1 \\
\; &\Longleftrightarrow\;
\sigma_f(x)=1\ \text{ and }\ \sigma_f(y)=-1.
\end{align*}
For the edge $xy$, we can see that
\[
I(z,x,y)\in\{x,y\},
\qquad
I(z,x,y)=x \Longleftrightarrow \sigma_e(z)=1,
\tag{$\ast\ast$}
\]
since $m(z,x,y)=x$ iff $x\in I(z,y)$ iff $d(z,y)=d(z,x)+1$.
Assume $e\Theta f$. Then by $(\ast)$ we have $\sigma_f(x)=1$ and $\sigma_f(y)=-1$, hence
$I(x,u,v)=u,I(y,u,v)=v$.
Using the standard median identity, we have
\[
I(z,u,v)=I\bigl(I(z,x,y),\,I(x,u,v),\,I(y,u,v)\bigr),
\]
we obtain $I(z,u,v)=I\bigl(I(z,x,y),u,v\bigr)$.
Therefore $I(z,u,v)=u$ iff $I(z,x,y)=x$, and by $(\ast\ast)$ this is equivalent to
$\sigma_f(z)=1$ iff $\sigma_e(z)=1$. Since $\sigma_e(z),\sigma_f(z)\in\{\pm 1\}$, it follows that
$\sigma_f(z)=\sigma_e(z)$ for all $z$, i.e. $\sigma_f=\sigma_e$.
Conversely, if $\sigma_e=\sigma_f$, then in particular $\sigma_f(x)=\sigma_e(x)=1$ and
$\sigma_f(y)=\sigma_e(y)=-1$, so $(\ast)$ implies $e\Theta f$.
\end{claimproof}

This proves the claim. Since $\Theta$ is equality of the maps $\sigma_e$, it is reflexive,
symmetric, and transitive, hence an equivalence relation on $\vec E(M)$.
\end{proof}

Given $\mathfrak h \in \vec {\mathcal H}(M)$, we denote by $\mathfrak h^r$ its \defin{reversal}. That is, 
$
\mathfrak h^r = \{(y,x) : (x,y) \in \mathfrak h\}
$. 
It is immediate from the definition of $\Theta$ that $\mathfrak h^r$ is also an oriented $\Theta$-class. We can also define unoriented $\Theta$-classes, slightly abusing notation and overloading the $\Theta$-relation. 

\begin{defn}[Unoriented $\Theta$-classes]
    Let $M$ be a median graph. Let $e =xy, f = uv \in E(M)$ be (unoriented) edges. Define the equivalence relation
\[
e\;\Theta f
\quad\Longleftrightarrow\quad
d(x,u) + d(y,v) \neq d(x,v) + d(y,u).
\]
The equivalences classes of $\Theta$ on $E(M)$ are called \defin{unoriented $\Theta$-classes}. Denote by $\mathcal H(M)$ the set of unoriented $\Theta$-classes of $M$. 
\end{defn}

Note that given unoriented edges $xy, uv\in E(M)$, we have that the following:
$$xy \Theta uv \textit{ if and only if } (x,y) \Theta (u,v) \textit{ or } (x,y) \Theta (v,u)$$ 
We leave it as an exercise to the reader to verify that this also defines an equivalence relation on $E(M)$ with a natural 2-1 map $\vec{\mathcal H}(M) \to \mathcal H(M)$, descended from the orientation-forgetting map $\vec E(M) \to E(M)$. Throughout this paper, we may simply refer to \emph{$\Theta$-classes}, without mention to whether they are oriented or unoriented, provided this is clear from context.
It is easy to see that if a group acts on a median graph, then this action preserves the $\Theta$-classes.

An important property of $\Theta$-classes is that each $\Theta$-class is a tight cut with very nice geometric properties.
More precisely, we have the following well-known statement; see \cite[\S11.2--12.1]{hammack2011handbook} for proofs.

\begin{prop}\label{prop:halfspaces-convex}
    Let $M$ be a median graph, and $\mathfrak h = [(u,v)] \in \vec {\mathcal H}(M)$. Then $M \sm \mathfrak h$ contains exactly two connected components, which are the subgraphs induced by 
    $$
    \mathfrak h_+ := \{w \in V(M) : d(w,u) < d(w,v)\}, \ \ \mathfrak h_- := \{w \in V(M) : d(w,u) > d(w,v)\}.
    $$
    Furthermore, the complementary components $M[\mathfrak h_\pm]$ are convex subgraphs of $M$.
\end{prop}

The subgraphs $M[\mathfrak h_\pm]$ are called \defin{halfspaces}. We may sometimes abuse notation and refer to $\mathfrak h_+$ itself as a halfspace. This notation for $\mathfrak h_\pm$ will follow us throughout this paper. Note that \Cref{prop:halfspaces-convex} implies that a geodesic in $M$ crosses any given $\Theta$-class at most once. 

Given non-empty subsets $S, S' \subset V(M)$, we say $\mathfrak h \in \vec {\mathcal H}(M)$ \defin{separates} $S$ and $S'$ if $S \subset \mathfrak h_{\pm}$ and $S' \subset \mathfrak h_\mp$. The separation properties of the $\Theta$-classes can be used to completely describe the metric of a median graph, and distances between convex subgraphs. This is summarised by the following quantitative form of the `Kakutani separation property' satisfied by median graphs. For a proof, combine \cite[Thm.~2.14]{haglund2023isometries} and \cite[Prop.~2.7]{genevois2021coning}. 

\begin{thm}[Kakutani separation property]\label{thm:kakutani}
    Let $M$ be a median graph. Let $C, C' \subset M$ be convex subgraphs. Then
    $$
    d(C, C') = \Big|\{\mathfrak h \in \mathcal H(M) : \text{$\mathfrak h$ separates $C$ and $C'$}\}\Big|.
    $$
    In particular, disjoint convex subgraphs are separated by some $\Theta$-class. 
\end{thm}



In trees, $\Theta$-classes correspond to edges. In particular, they naturally ordered. In a general median graph, a key difference is that $\Theta$-classes are allowed to `cross'.

\begin{defn}[Crossing and parallel $\Theta$-classes]
    Let $M$ be a median graph. Let $\mathfrak h, \mathfrak h' \in \vec {\mathcal H}(M)$. We say that $\mathfrak h$ \defin{crosses} $\mathfrak{h}'$ if
    $$
    \mathfrak h_+ \cap \mathfrak h' _+, \ \ \mathfrak h_+ \cap \mathfrak h' _-,   \ \ \mathfrak h_- \cap \mathfrak h' _+,   \ \ \mathfrak h_- \cap \mathfrak h' _-
    $$
    are all non-empty. If $\mathfrak h$ and $\mathfrak h'$ do not cross, we say they are \defin{parallel}. 
\end{defn}

Of course, it also makes sense to refer to unoriented $\Theta$-classes as crossing or parallel, and we will do so at whim. Note that there is a natural partial ordering (in fact, a pocset structure; see \Cref{sec:sageev}) on $\vec {\mathcal H}(M)$, given by 
$$
\mathfrak h \leq \mathfrak h' \Leftrightarrow \mathfrak h_+ \subseteq \mathfrak h_+'.
$$
In this language, two oriented $\Theta$-classes $\mathfrak h$, $\mathfrak h'$ cross exactly when $\mathfrak h'$ is incomparable with both $\mathfrak h$ and $\mathfrak h^r$ within this partial ordering.

We conclude this section by introducing one final piece of terminology.

\begin{defn}
    A median graph $M$ is said to be \defin{crossing-finite} (resp. \defin{crossing-bounded}) if all sets of pairwise-crossing $\Theta$-classes of $M$ are finite (and uniformly bounded). 
\end{defn}

Recall that every median graphs can be canonically decomposed into cube graphs (being the 1-skeleton of a CAT(0) cube complex). The crossing-finite hypothesis essentially asserts that there is no `infinite cube' in this decomposition. The crossing-bounded property puts a uniform bound on the dimension of a cube appearing in this decomposition.

\begin{rem}\label{rem:dimension}
    If $M$ is a median graph, then the maximum cardinality of any set of pairwise crossing-$\Theta$-classes is sometimes referred to as the \defin{(cubical) dimension} of $M$ in the literature, denoted $\dim(M)$. As such, crossing-bounded median graphs are often called \defin{finite-dimensional}.  
\end{rem}


\section{Median decompositions and their induced separations}\label{sec:med-decomp}

We now address the concept of a \emph{median decomposition} of a graph, adopting the notation and definitions from Stavropoulos \cite{stavropoulos2015medianwidth,stavropoulos2016graph}.

\subsection{Median decompositions}

We have the following definition.

\begin{defn}[Median decompositions]
A \defin{median decomposition} of a graph $G$ is a pair
$(M,\beta)$, where $M$ is a median graph, and $\beta\colon V(M)\to \mathcal P\bigl(V(G)\bigr)$ assigns to every node  $u\in V(M)$ a \defin{bag} $\beta(t)\subseteq V(G)$
such that
\begin{enumerate}[label=(M\arabic*)]
\item\label{m1} $\displaystyle \bigcup_{u\in V(M)} \beta(u) \;=\; V(G)$;
\item\label{m2} for every edge $xy\in E(G)$ there exists
      $u\in V(M)$ with $\{x,y\}\subseteq \beta(u)$;
\item\label{m3} for every $v\in V(G)$ the set
      $\beta^{-1}(v)\coloneqq\{u\in V(M) :  v\in \beta(u)\}$
      induces a convex subgraph of $M$.
\end{enumerate}

The sets $\beta^{-1}(v)$ are called the \defin{fibres} of $(M,\beta)$.
Given an edge $st\in E(M)$, the intersection
$\beta(s)\cap \beta(t)$ is called an \defin{adhesion set} of the decomposition.
The sets $ \{\beta(u) :  u\in V(M)\} $ for every $ u \in V(M) $ are called the \defin{bags} of $ (M, \beta) $, and the induced subgraphs $ G[\beta(u)] $ are the \defin{parts} of $ (M, \beta) $. 

The \defin{width} of $(M,\beta)$ is defined as as $\mathrm{width}(M,\beta) = \sup_{u \in V(M)} |\beta(u)|$. 
Then, the \defin{median-width}\footnote{When defining \emph{tree-width}, one normally takes $\mathrm{tw}(G) = \inf \mathrm{width}(T,\beta) - 1$. The point of this convention is to ensure that trees are exactly the graphs of tree-width 1. As remarked by Stavropoulos in \cite{stavropoulos2015medianwidth}, this convention does not have the same effect for median graphs, as any triangle-free graph admits a median decomposition of width 2; see \Cref{thm:cliquebags}. As such, the `$-1$' convention is abandoned.} 
of $ G $, denoted by $ \mathrm{mw}(G) $,
is defined as:
\[
\mathrm{mw}(G) = \inf \mathrm{width}(M,\beta),
\]
where the infimum is taken over all median-decompositions $ (M,\beta) $ of $ G $.

Suppose $G$ is equipped with an action by a group $\Gamma$. 
Then $(M,\beta)$ is called \defin{$\Gamma$-canonical} if $\Gamma$ acts on $M$ such that $g\cdot \beta(v)=\beta(g\cdot v)$, for every $v \in V(M)$, $g \in \Gamma$.
\end{defn}

We invite the reader to compare this definition with the usual definition of a tree decomposition. It is essentially the same, with the exception of \ref{m3}. For tree decompositions, one usually asks that fibres be \emph{connected}. However, recall that the convexity and connectivity coincide for subgraphs of trees. In particular, taking $M$ to be a tree in the above definition precisely recovers the usual definition of a tree decomposition. 

\begin{exa}\label{exa:example-nonexample}
    Consider the path graph $G$ of length 3, with vertices $V(G) = \{1,2,3,4\}$. In~\Cref{fig:example-nonexample}, two decompositions of this graph are depicted. Only one is a well-formed median decomposition. 
    \begin{figure}[h]
        \centering
        \subfloat[\centering The graph $G$.]{{\includegraphics[scale=0.5]{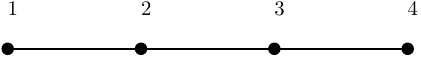}}}
        \qquad
        \subfloat[\centering Example of a median decomposition.]{{\includegraphics[scale=0.5]{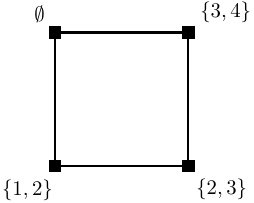} }}%
        \qquad
        \subfloat[\centering  Non-example.]{{\includegraphics[scale=0.5]{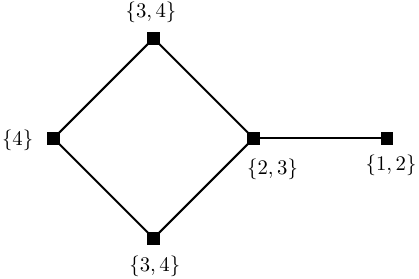} }}%
        \caption{}
        \label{fig:example-nonexample}
    \end{figure}
    
\end{exa}

\subsection{Separations induced by \texorpdfstring{$\Theta$}{Theta}-classes}

It is well-known that, given a tree decomposition $(T, \beta)$,  oriented edges of $T$ give rise to separations of $G$. This runs as follows. Each oriented edge $e \in \vec E(T)$ induces an edge-cut $E[U,W]$ of $T$. Then, the separation of $G$ induced by $e$ can be defined to be 
$$
(A_e, B_e) := \Bigg( \bigcup_{s\in U}\beta(s), \bigcup_{t\in W}\beta(t)\Bigg).
$$
Repeating this for every oriented edge in $T$ gives rise to a symmetric set of separations $\Sigma_{T,\beta}$, which morally encodes the same structural data as $(T,\beta)$, though less organised. 

In this section, we remark that the same phenomenon makes sense in median decompositions, via the $\Theta$-classes of the median graph. These separations, in turn, serve as labels for the $\Theta$-classes. 

\begin{defn}
    Let $G$ be a graph and $(M, \beta)$ be a median decomposition of $G$. Let $\mathfrak h \in \vec {\mathcal H}(M)$ be an oriented $\Theta$-class. Let $E[\mathfrak h_+, \mathfrak h_-]$ be the corresponding edge-cut in $M$. Then the \defin{separation induced by $\mathfrak h$}, denoted $(A_{\mathfrak h}, B_{\mathfrak h})$, is defined to be
    $$
    A_{\mathfrak h} \coloneqq \bigcup_{{u}\in \mathfrak h_+}\beta(u), 
    \qquad
    B_{\mathfrak h} \coloneqq \bigcup_{v\in \mathfrak h_-}\beta(v).
    $$
    The \defin{set of separations dual to $(M,\beta)$} is then defined as
    $$
    \Sigma_{M,\beta} = \{(A_{\mathfrak h}, B_{\mathfrak h}) : \mathfrak h \in \vec {\mathcal H}(M)\}. 
    $$
\end{defn}

Given a median decomposition $(M,\beta)$, it is easy to see that the dual set of separations $\Sigma_{M,\beta}$ is symmetric. However, it certainly need not be the case that every $(A_\mathfrak h, B_\mathfrak h)$ be an essential separation. In particular, $\Sigma_{M,\beta}$ is not necessarily a separation system.
If $\Sigma_{M,\beta}$ \emph{is} a separation system, we refer to it as the \defin{dual separation system}. 
We immediately record the following easy observation, as it is quite useful.

\begin{lem}\label{lem:fibre-halfspace-char}
    Let $G$ be a graph and $(M, \beta)$ a median decomposition. Let $x \in V(G)$ and $\mathfrak h \in \vec {\mathcal H}(M)$. Then $x \in A_\mathfrak h$ if and only if $\beta^{-1}(x) \cap \mathfrak h_+ \neq \emptyset$. 
\end{lem}

The following describes a link between $\Sigma_{M,\beta}$ and $\beta$. In particular, it shows that $\beta$ is completely determined by the assignment $\vec {\mathcal H}(M) \to \Sigma_{M,\beta}$, $\mathfrak h \mapsto (A_{\mathfrak h}, B_{\mathfrak h})$. 

\begin{prop}\label{prop:bags-are-intersections}
    Let $G$ be a graph and $(M, \beta)$ be a median decomposition of $G$ with dual set of separations $\Sigma = \Sigma_{M,\beta}$. Let $v \in V(M)$, and let $S_v \subset \Sigma$ denote the set 
    $$
    S_v = \{(A_\mathfrak h, B_\mathfrak h) : \text{$\mathfrak h \in \vec {\mathcal H}(M)$, $v \in \mathfrak h_+$}\}.
    $$
    Then
    $$
    \beta(v) = \bigcap_{(A,B) \in S_v} A.
    $$
\end{prop}

\begin{proof}
    Write $Z = \bigcap_{(A,B) \in S_v} A$. It is immediate from the definitions of $S_v$ and $\Sigma_{M,\beta}$ that $\beta(v) \subset Z$. Conversely, let $y \in Z$. Then for every $\mathfrak h \in \vec {\mathcal H}(M)$ such that $v \in \mathfrak h_+$, we have that  $\beta^{-1}(y)$ intersects $\mathfrak h_+$. Suppose that $v \not\in \beta^{-1}(y)$, then by the Kakutani separation property (\ref{thm:kakutani}) there exists $\mathfrak h \in \vec {\mathcal H}(M)$ such that $v \in \mathfrak h _+$ and $\beta^{-1}(y) \subset \mathfrak h_-$, which is a contradiction. 
    Thus, $v \in \beta^{-1}(y)$, and so $y \in \beta(v)$. Since $y$ was arbitrary, it follows that $Z \subset \beta(v)$. 
\end{proof}

In \Cref{sec:dual-meddecomp}, we will describe the reverse of this construction. That is, given a sufficiently sensible separation system $\Sigma$, how to construct a median decomposition $(M,\beta)$ such that $\Sigma_{M,\beta} = \Sigma$. 

\subsection{Reduced and crossing-faithful median decompositions}

Recall that a tree decomposition is usually said to be \defin{reduced} if no bag is contained in another bag. This definition does not immediately port over to median decompositions. Indeed, many perfectly sensible median decompositions feature empty bags.
Instead, the `correct' definition of a reduced median decomposition makes use of the dual set of separations, as follows.

\begin{defn}[Reduced median decomposition]\label{def:reduced}
    Let $G$ be a graph and $(M,\beta)$ be a median decomposition. We say that $(M,\beta)$ is \defin{reduced} if the following hold:
    \begin{enumerate}[label=(Red\arabic*)]
        \item\label{red1} Every induced separation $(A_\mathfrak h, B_\mathfrak h)$ is essential. 

        \item\label{red2} The natural map $\vec {\mathcal H}(M) \to \Sigma$, $\mathfrak h \mapsto (A_{\mathfrak h}, B_{\mathfrak h})$ is injective.
    \end{enumerate}
    A median decomposition which satisfies \ref{red1} but not necessarily \ref{red2} is called \defin{weakly reduced}.\footnote{What we call \emph{weakly reduced} has previously been referred to as \emph{regular} in the literature, for example in \cite{jacobs2025canonical}.} 
\end{defn}

\begin{rem}
     We remark that \Cref{def:reduced} recovers the usual notion of a reduced tree decomposition when the median graph $M$ is a tree. 
\end{rem}

 It is immediate that the dual set of separations $\Sigma_{M,\beta}$ is a separation system (i.e. contains only essential separations) if and only if $(M,\beta)$ is weakly reduced. This can also be characterised as follows.

\begin{prop}\label{prop:weakly-reduced-fibre-char}
    Let $G$ be a graph and $(M,\beta)$ be a median decomposition. Then $(M,\beta)$ is weakly reduced if and only if there does not exist a convex proper subgraph $C \subset M$, such that for all $x \in V(G)$, $\beta^{-1}(x) \cap V(C) \neq \emptyset$.  
\end{prop}

\begin{proof}
    Suppose $(M, \beta)$ is not weakly reduced, so there exists $\mathfrak h \in \vec {\mathcal H}(M)$ such that $A_\mathfrak h = V(G)$. In particular, we may take $C = M[\mathfrak h_+]$. Conversely, suppose such a $C \subset M$ exists. Let $v \in V(M) \sm V(C)$. By \Cref{thm:kakutani}, there exists $\mathfrak h \in \vec {\mathcal H}(M)$ such that $v \in \mathfrak h_-$ and $V(C) \subset \mathfrak h _+$. In particular, since every fibre intersects $V(C)$, we have that $A_\mathfrak h = V(G)$. This implies the proposition.
\end{proof}

\begin{rem}\label{rem:weak-reduction}
    Given a median decomposition $(M,\beta)$ which is not weakly reduced, one can iteratively `reduce' this decomposition by collapsing $\Theta$-classes which induce inessential separations. More precisely, let $(M, \beta)$ be a median decomposition of a graph $G$, and let $\Sigma = \Sigma_{M,\beta}$. We define
    $$
    \Sigma^{\mathrm{ess}} := \{((A,B) \in \Sigma : \text{$(A,B)$ is essential}\},
    $$
    and say that $\mathfrak h \in \vec {\mathcal H}(M)$ is \defin{essential in $(M,\beta)$} if $(A_\mathfrak h, B_{\mathfrak h})$ is essential. We then define a new graph $\widehat M = M/\sim$ as a quotient graph of $M$, where $u \sim v$ if and only if every $\Theta$-class separating $u$ and $v$ is not essential. It is well-known that $\widehat M$ is also a median graph, obtained from $M$ by `collapsing' the inessential $\Theta$-classes. We can then define a new median decomposition $(\widehat M, \widehat \beta)$ via
    $$
    \widehat \beta([v]) := \bigcup_{u \in [v]} \beta(u).
    $$
    This defines a weakly reduced median decomposition with dual separation system exactly $\Sigma^{\mathrm{ess}}$, called the \defin{weak reduction of $(M,\beta)$}. 
    One can apply similar reasoning to `fully reduce' a median decomposition, by also collapsing $\Theta$-classes labelled by redundant induced separations, though this requires certain choices to be made as to which $\Theta$-classes to collapse, and as such is less canonical.
\end{rem}

There is another sense in that a median decompositions can be somehow `degenerate', which doesn't apply for tree decompositions. Before we describe this, we first state record the following proposition, which describes how nesting and crossing of $\Theta$-classes informs the structure of the dual set of separations.
\begin{prop}\label{prop:nest-cross-induced-seps}
    Let $G$ be a graph. Let $(M,\beta)$ be a median decomposition, and fix $\mathfrak h, \mathfrak h' \in \vec{\mathcal H}(M)$. 
    \begin{enumerate}
        \item If $\mathfrak h_+ \subset \mathfrak h'_+$, then  $(A_{\mathfrak h}, B_{\mathfrak h}) \leq (A_{\mathfrak h'}, B_{\mathfrak h'})$. In particular, if $M$ is crossing-finite (crossing-bounded), then $\Sigma_{M,\beta}$ is crossing-finite (crossing-bounded).

        \item Suppose further that $(M, \beta)$ is weakly reduced. If $(A_{\mathfrak h}, B_{\mathfrak h}) < (A_{\mathfrak h'}, B_{\mathfrak h'})$, then either $\mathfrak h$ and $\mathfrak h'$ cross, or $\mathfrak h_+ \subset \mathfrak h_+'$. 
    \end{enumerate}
\end{prop}

\begin{proof}
    To ease notation,  write $(A,B) = (A_{\mathfrak h}, B_{\mathfrak h})$, $(A', B') = (A_{\mathfrak h'}, B_{\mathfrak h'})$.
    We have that (1) follows immediately from the construction of $\Sigma_{M,\beta}$. To see (2), suppose that $(A, B) < (A', B')$. Assume without loss of generality that $\mathfrak h$ and $\mathfrak h'$ are nested. This leaves us with four possibilities four how they nested, three of which we must rule out.

    First, suppose that $\mathfrak h'_+ \subset \mathfrak h_+$. Then, by (1), we have that $(A', B') \leq (A, B) < (A', B')$, which is clearly nonsense. 
    Suppose instead that $\mathfrak h_+ \subset \mathfrak h'_-$. Then $(A,B) \leq (B',A') < (B,A)$, and so $B = V(G)$ which contradicts the assumption that $(M,\beta)$ is weakly reduced.
    Finally, suppose that $\mathfrak h'_- \subset \mathfrak h_+$. Then $(B',A') \leq (A,B) < (A',B')$, and so $A' = V(G)$, which is also a contradiction.
    We deduce by process of elimination that $\mathfrak h_+ \subset \mathfrak h_+'$. This completes the proof of (2).
\end{proof}

Note that it is possible for two crossing $\Theta$-classes to induce nested separations, as can be seen in the following example.

\begin{exa}\label{exa:not-crossing-faithful}
    Let $G$ denote the path graph of length 3, with vertex set $V(G) = \{1,2,3, 4\}$. We return to the example of a median decomposition $(M,\beta)$ of $G$ given in \Cref{exa:example-nonexample}.  We have that $M$ contains exactly two unoriented $\Theta$-classes, which cross. The dual separations are labelled in~\Cref{fig:crossing-failure}. Note that their associated separations are distinct and nested. 
 \begin{figure}[h]
        \centering
        \includegraphics[scale=0.7]{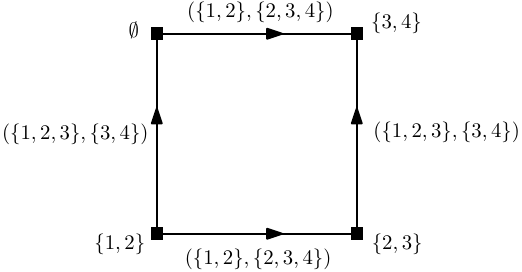}
        \caption{The median-decomposition of $P_4$}
        \label{fig:crossing-failure}
    \end{figure}\end{exa}

This motivates the following definition. 

\begin{defn}\label{def:crossing-faithful}
    Let $G$ be a graph and $(M,\beta)$ be a median decomposition. Then $(M,\beta)$ is said to be \defin{crossing-faithful} if, for all pairs $\mathfrak h, \mathfrak h' \in \vec{\mathcal H}(M)$ of crossing $\Theta$-classes, we have that the induced separations $(A_{\mathfrak h}, B_{\mathfrak h})$ and $(A_{\mathfrak h'}, B_{\mathfrak h'})$ also cross. 
\end{defn}

We will see in \Cref{sec:duality} that a reduced, crossing-faithful median decomposition is completely determined by its dual system of separations.

We conclude this section with the following proposition, which describes how median decompositions often give rise to `nice' dual sets of separations.

\begin{prop}\label{prop:weakly-reduced-crossing-finite-gives-nice-seps}
    Let $G$ be a graph and $(M,\beta)$ be a weakly reduced median decomposition. Let $\Sigma = \Sigma_{M,\beta}$ be the dual system of separations. Suppose $M$ is crossing-finite, or $(M,\beta)$ is crossing-faithful. Then $\Sigma$ is discrete and satisfies the ECC. 
\end{prop}

\begin{proof}
    Let us suppose for the sake of a contradiction that $\Sigma = \Sigma_{M,\beta}$ does not satisfy the ECC. Let 
    $$
    x \in \bigcap_{n=1}^\infty (A_n \sm B_n),
    $$
    where $(A_1, B_1) > (A_2, B_2) > \ldots$ is an infinite, strictly descending chain. 
    For each $i \geq 1$, fix $\mathfrak h^i \in \vec {\mathcal H}(M)$ such that $(A_i, B_i) = (A_{\mathfrak h ^i}, B_{\mathfrak h ^i})$. 
    Since $(M,\beta)$ is weakly reduced, by \Cref{prop:nest-cross-induced-seps} we have that the only possible relationships between $\mathfrak h^{i}$ and $\mathfrak h^{i+1}$ is either that these two $\Theta$-classes cross, or $\mathfrak h^i_+ \supset \mathfrak h^{i+1}_+$. 
    
    By~\Cref{lem:fibre-halfspace-char}, we have that $\beta^{-1}(x) \subset \mathfrak h^i_+$ for every $i \geq 1$. 
    By Ramsey's theorem, $S = \{\mathfrak h^i_+\}_i$ contains either an infinite chain or an infinite anti-chain. If $S$ contains an infinite chain, then $\bigcap_{i} \mathfrak h^i_+ = \emptyset$, and so $\beta^{-1}(x) = \emptyset$, which cannot happen by \ref{m1}. 
    If $M$ is crossing-finite, or $(M,\beta)$ is crossing-faithful, then $S$ cannot contain an infinite anti-chain. It follows that $\Sigma$ satisfies the ECC.

    We now show that $\Sigma$ is discrete. Suppose not, so there exists separations $(A,B) < (A', B')$ such that the interval
    $$
    I := \{(C,D) \in \Sigma : (A_1, B_1) < (C,D) < (A_2, B_2)\}
    $$
    is infinite. Fix $\mathfrak h, \mathfrak h'$ such that $(A,B) = (A_{\mathfrak h}, B_{\mathfrak h})$, $(A', B') = (A_{\mathfrak h'}, B_{\mathfrak h'})$. If $(M, \beta)$ is crossing-faithful, then we immediately deduce that $\mathfrak h_+$ and $\mathfrak h'_-$ are disjoint, and there exists infinitely many $\mathfrak h''$ such that $\mathfrak h_+ \subset \mathfrak h''_+$ and $\mathfrak h'_- \subset \mathfrak h_-''$. By \Cref{thm:kakutani}, this means that $\mathfrak h_+$ is infinitely far from $\mathfrak h'_-$, a contradiction.
    Suppose instead that $M$ is crossing-finite.
    By Ramsey's theorem, $I$ contains either an infinite chain or an infinite anti-chain. Note that we cannot have that $(D,C) \leq (C',D')$ for any $(C,D), (C',D') \in I$, as this would imply that $C' \supset D'$, meaning $(C',D')$ is not proper. Therefore, an anti-chain would consist of crossing elements. 
    By \Cref{prop:nest-cross-induced-seps}, since $M$ is crossing-finite we have that $\Sigma$ is crossing-finite, and so $I$ contains an infinite ascending or descending chain, say $(A_1, B_1) > (A_2, B_2) > \ldots$ without loss of generality (an identical argument applies when the chain is ascending). 
    Note that $A \subset \bigcap_i A_i$, and $B \supset \bigcup_i B_i$. This implies that $A \sm B \subset \bigcap_i (A_i \sm B_i)$. But, since $\Sigma$ is known to satisfy the ECC, the right-hand side is empty. In particular, $A \subset B$, which means that $(A,B)$ is not an essential separation. This contradicts the assumption that $(M,\beta)$ is weakly reduced, and we are done.
\end{proof}

\begin{rem}
    If $(M,\beta)$ is a median decomposition which is not weakly reduced, but either $M$ is crossing-finite or $(M,\beta)$ is crossing-faithful, then by passing to the weak reduction $(\widehat M,\widehat \beta)$ described in \Cref{rem:weak-reduction}, we deduce that the essential subset $\Sigma^{\mathrm{ess}}_{M,\beta}$ is a discrete, ECC system of separations, even if $\Sigma_{M,\beta}$ itself is poorly behaved.
\end{rem}

The following example illustrates the additional hypothesis cannot be dropped from \Cref{prop:weakly-reduced-crossing-finite-gives-nice-seps}.

\begin{exa}\label{exa:dual-seps-not-ECC}
    Let $G$ be the infinite star graph with vertices $V(G) = \{b, 1,2,3, \ldots, \omega\}$, where $b \in V(G)$ is the central base vertex and every other vertex is a leaf. We will consider a separation system of $G$ similar to that which appears in \Cref{exa:no-DCC-principal-filter}. 
    That is, let $\Sigma = \{(A_n, B_n), \  (B_n, A_n) : n \geq 1\} \cup \{(A_\omega, B_\omega), (B_\omega, A_\omega)\}$  where
    \begin{align*}
        A_n &= \{b, 1, \ldots, n\}, \ \  \ \  \ B_n = \{\omega, b, n+1, n+2, \ldots\} \\
        A_\omega &= \{b, 1, 2, \ldots\}, \ \ \ \ \ \ B_\omega = \{\omega, b\}.
    \end{align*}
    
    Consider the following median decomposition $(M,\beta)$ of $G$.
    The median graph $M$ is the infinite cube, identified with the standard Cayley graph of the group $\Gamma = \bigoplus_{n \geq 0} \mathbb Z_2$. That is, the Cayley graph associated with the generating set consisting of standard unit vectors. The bag-map $\beta$ is defined as follows:
    \begin{enumerate}
        \item The base-vertex $b$ appears in every bag.

        \item The distinguished vertex $\omega \in V(G)$  appears in $\beta(\mathbf v)$ if and only if $\mathbf v$ contains a 1 in the $0$-th position. 

        \item The vertex $n \in V(G)$ appears in $\beta(\mathbf v)$ if and only if $\mathbf v$ contains a 1 in the $n$-th position. 
    \end{enumerate}
    It is easy to see verify that $(M,\beta)$ is a reduced median decomposition, and that $\Sigma_{M,\beta}$ is exactly $\Sigma$ described above. In particular, it does not satisfy the ECC, nor is it discrete.
\end{exa}




\section{Sageev's construction}\label{sec:sageev}

In this section, we recount Sageev's construction of the dual median graph. This is all fairly classical, though proofs are included for the sake of completeness and exposition.

\subsection{Pocsets and ultrafilters}

We first set up some basic terminology. 
A \defin{pocset} is a set $P$ equipped with a partial ordering $\leq$ and an order reversing involution $s \mapsto s^r$, such that  $s \not\leq s^r$ for all $s \in P$. 
The main example of interest to this paper is that if $G$ is a graph and $\Sigma$ is a separation system, then $\Sigma$ is naturally a pocset. Another important example is, given a median graph $M$, the set $\vec {\mathcal H}(M)$ of oriented $\Theta$-classes admits a natural pocset structure.

 We say that $s_1, s_2 \in P$ \defin{cross} if 
 $$
 s_1 \not\leq s_2, \ \ s_2 \not\leq s_1, \ \ s_1 \not\leq s_2^r, \ \ \text{and} \ \ s_2^r \not\leq s_1.
 $$ 
 The pocset $P$ is said to be \defin{crossing-finite} if every anti-chain in $P$ is finite. 
 We say that $P$ is \defin{discrete} if for all $s, s' \in P$, the set $\{t \in \Sigma : s \leq t \leq s'\}
$ is finite. Note that this agrees with every earlier use of the terminology `crossing-finite' and `discrete' in this paper.

\begin{defn}[Ultrafilters]
Let $P$ be a pocset.
An \defin{ultrafilter} on $P$ is a subset $\mathcal{U} \subset P$ that satisfies the following:
\begin{enumerate}[label=$(\mathrm{U}\arabic*)$]
\item\label{p1} For every
 separation $s \in P$, exactly one of $s$ or $s^r$ lies in $\mathcal U$.
\item\label{p2}  If $s\in\mathcal{U}$ and $s\le t$, then $t\in\mathcal{U}$.
\end{enumerate}
We say that $\mathcal U$ satisfies the \defin{descending chain condition} if every descending chain in $\mathcal U$ is finite, and call $\mathcal U$ a \defin{DCC ultrafilter}. 
\end{defn}

\begin{rem}\label{rem:ultrafilters-on-hyperplanes}
    It is helpful to observe that, given a median graph $M$, there is a natural bijection
    $$
    V(M) \ \ \longleftrightarrow \ \ \{\text{DCC ultrafilters on $\vec {\mathcal H}(M)$}\}.
    $$
    via the map $v \mapsto \{\mathfrak h : v \in \mathfrak h_+\}$. Also, given a DCC ultrafilter $\mathcal U \subset \vec {\mathcal H}(M)$, it is easily checked that
    $$
    \Big|\bigcap_{\mathfrak h \in \mathcal U} \mathfrak h_+ \Big| = 1,
    $$
    giving the other direction.
\end{rem}

We now record three basic lemmas regarding ultrafilters. 

\begin{lem}\label{lem:flipping-elements}
    Let $P$ be a pocset. 
    Let $\mathcal{U}$ be a ultrafilter on $ P $, $ s \in\mathcal{U}$.  
    Define
    \[
    \mathcal{U}'\;\coloneqq\;\bigl(\mathcal{U}\sm\{ s \}\bigr)\cup\{s^r\}.
    \]
    Then $\mathcal{U}'$ is a ultrafilter if and only if $ s $ is $\le$‑minimal in $\mathcal{U}$. 
\end{lem}

\begin{proof}
    
    First, assume $\mathcal{U}'$ is an ultrafilter but $s$ is not $\le$‑minimal in $\mathcal{U}$.  
Choose $ t \in\mathcal{U}\sm\{s\}$ with $ t <s$.  
Because $ t $ was not removed, we still have $ t \in\mathcal{U}'$.  
Since $\mathcal{U}'$ satisfies property~\ref{p2}, the inequality $  t <s $ forces $s\in\mathcal{U}'$, contradicting $ s \notin\mathcal{U}'$.  
Hence $ s $ must be $\le$‑minimal.

For the backward implication, suppose $ s $ is $\le$‑minimal in $\mathcal{U}$. 
We now show that $\mathcal{U}'$ defined as above is an ultrafilter. Indeed, it is immediate from the construction of $\mathcal U'$ that $\mathcal U'$ satisfies \ref{p1}.
We now verify \ref{p2}.
Let $ t \in\mathcal{U}'$ and $ t' \in P $ with $ t < t' $.  
We distinguish two cases.
Suppose first that $ t \neq s^r $.  
        Then $ t \in\mathcal{U}\sm\{ s \}$, so $ t' \in\mathcal{U}$ by \ref{p2} for $\mathcal{U}$.  
        If $ t' \neq s $, we have $ t' \in\mathcal{U}'$ directly.  
        The equality $ t' = s $ is impossible because $ t <  t' $ would violate the minimality of $ s $.
Suppose instead then that $ t = s^r $.  
        Since $ s^r <  t' $, taking inverses gives $t'^r<  s $  
        Because $ s $ is minimal, $t'^r\notin\mathcal{U}$, forcing $ t' \in\mathcal{U}$.  
        We must have that $ t' \neq s $, as otherwise we get the inequality $s^r < s$, which is not allowed in a pocset. It follows that $ t' \in\mathcal{U}'$.
Thus $\mathcal{U}'$ satisfies both \ref{p1} and \ref{p2}, and is an ultrafilter on $ P $. 
\end{proof}

\begin{lem}
    Let $P$ be a pocset. 
    Let $\mathcal U$, $\mathcal V$ be ultrafilters on $P$ such that $\mathcal U \triangle \mathcal V$ is finite. Then $\mathcal U$ is DCC if and only if $\mathcal V$ is DCC.
\end{lem}

\begin{proof}
    Since $\mathcal U \triangle \mathcal V$ is finite, an infinite descending chain in $\mathcal U$ would obviously yield one in $\mathcal V$, and vice versa. 
\end{proof}

\begin{lem}\label{lem:symmetric-diff}
Let $P$ be a discrete, crossing-finite pocset. 
Let $\Ucal$, $\mathcal V$ be DCC ultrafilters on $P$.
Then $\Ucal \triangle \Vcal$ is finite. Moreover, every $\le$-minimal element in $\Ucal \sm \Vcal $ is also $\le$-minimal in $\Ucal$. 
\end{lem}

\begin{proof}
Write 
$W \coloneqq \mathcal{U} \sm \mathcal{V}$.
If $W$ were infinite then it would necessarily contain either an infinite chain or an infinite anti-chain (i.e. an infinite set of pairwise-crossing elements). Since $P$ is crossing-finite, $W$ must therefore contain an infinite (necessarily ascending) chain, say $\{s_i : i \in \mathbb N\}$ with $s_i < s_{i+1}$ for all $i \in \mathbb N$. But then $\{s_i^r : i \in \mathbb N\}$ is an infinite descending chain in $\mathcal V$, which contradicts the assumption that $\mathcal V$ is a DCC ultrafilter. 

Let $s$ be a $\le$-minimal separation in $W$. 
Then, $s^r \in \mathcal{V}$. 
Suppose, for the sake of contradiction, that there exists $t \in \mathcal{U} \sm W = \mathcal{U} \cap \mathcal{V}$ such that $ t <s$, so $ s^r<t^r$. 
Since $s^r \in \mathcal{V}$, this implies that $t^r \in \mathcal{V}$. 
Given that $t \in \mathcal{U}$ and $t^r \in \mathcal{V}$, it follows that $t \in W$, contradicting the $\le$-minimality of $s$ in $W$. 
Therefore, $s$ is a $\le$-minimal element of $\mathcal{U}$.
\end{proof}

\subsection{The dual median graph}

We are now ready to define the dual median graph, and verify its core properties. We will focus on discrete, crossing-finite pocsets, though some of this machinery does extend beyond this setting. 

\begin{defn}[The dual median graph]\label{def:M_sigma}
Let $G$ be a graph, let $P$ be a discrete, crossing-finite pocset.
We then define the graph
$M_P$
as follows:
\begin{enumerate}
\item The vertex set $V(M_P)$ is the set of all DCC ultrafilters on $P$.
\item Two vertices $\mathcal{U},\mathcal{V} \in V(M_P)$ are adjacent exactly when $|\mathcal{U} \triangle \mathcal{V}| = 2$.
Equivalently, $\mathcal{U}$ is adjacent to $\mathcal{V}$ if some $s\in P$ satisfies
\[
\mathcal{V}
\;=\;
\bigl(\mathcal{U} \sm \{s\}\bigr) \cup \{s^r\}.
\]
\item We label the oriented edge $(\mathcal U, \mathcal V) \in \vec E(M_P)$ with $s$, equipping $M_P$ with a \defin{labelling function}  $\ell : \vec E(M_P) \to P$. 
\end{enumerate}
The graph $M_P$ is called the \defin{median graph dual to $P$}. 
\end{defn}

There is much to verify here, most importantly we need to establish that $M_P$ is a connected median graph. The following proposition asserts that it is connected. 

\begin{prop}\label{prop:distance-formula}
    Let $P$ discrete, crossing-finite pocset. Then $M_P$ is connected. Given $\mathcal U, \mathcal V \in V(M_P)$, their distance in $M_P$ is given by the formula 
    $$
    d(\mathcal U, \mathcal V) = \frac 1 2 |\mathcal U \triangle \mathcal V|.
    $$
\end{prop}

\begin{proof}
    Let $\mathcal U$ and $\mathcal V$ be DCC ultrafilters on $P$. Write
$k \coloneqq |\mathcal{U}\triangle\mathcal{V}|$.
We have seen by \Cref{lem:symmetric-diff} that $k$ is finite. It is easy to see then that $k$ must be even.
We proceed by induction on the even integer $k$.
If $k=2$ then
$\mathcal{U}$ and $\mathcal{V}$ differ in exactly one ordered pair
$\{s,s^r\}$; hence they are adjacent, and obviously the single edge
$\gamma\colon\mathcal{U},\mathcal{V}$ is a geodesic between these two points.
Assume then that $k>2$ and that the claim holds for all
even values smaller than $k$.
Choose a $\le$-minimal separation
$s\in\mathcal{U}\sm\mathcal{V}$.
By \Cref{lem:symmetric-diff}, $s$ is also minimal in $\mathcal{U}$; hence
\Cref{lem:flipping-elements} implies that
$\mathcal{W}\coloneqq \bigl(\mathcal{U}\sm\{s\}\bigr)\cup\{s^r\}$
is a DCC ultrafilter adjacent to $\mathcal{U}$ in $M_P$.
Because exactly one ordered pair has been flipped,
$|\mathcal{W}\triangle\mathcal{V}|=k-2$.
By the induction hypothesis there exists a geodesic
$\gamma'$ from $\mathcal{W}$ to $\mathcal{V}$.
Concatenating the initial edge
$\mathcal{U}\mathcal{W}$ with $\gamma'$ yields a
$\bigl(\mathcal{U},\mathcal{V}\bigr)$-geodesic $\gamma$ in $M_P$.
\end{proof}

We now verify that $M_P$ is indeed a median graph. 

\begin{prop}\label{lem.it.is.median.graph}
Let $P$ be a discrete, crossing-finite pocset. 
Then $M_P$ is a median graph.
\end{prop} 

\begin{proof}
Let $\mathcal{U},\mathcal{V},\mathcal{W}\in V(M_P)$.
Then we shall show that the set
\[
I(\mathcal{U},\mathcal{V},\mathcal{W})
\;\coloneqq\;
I(\mathcal{U},\mathcal{V})\cap I(\mathcal{U},\mathcal{W})\cap I(\mathcal{V},\mathcal{W})
\]
has cardinality~$1$.
Let
\begin{equation}\label{eq:C-def}
\mathcal{C}
\;\coloneqq\;
(\mathcal{U}\cap\mathcal{V})
\;\cup\;
(\mathcal{U}\cap\mathcal{W})
\;\cup\;
(\mathcal{V}\cap\mathcal{W}).
\end{equation}
We first show that $\mathcal C$ is a DCC ultrafilter.
It is not hard to see that
\begin{equation}\label{eq:C-Demorgan}
\mathcal{C}
\;=\;
(\mathcal{U}\cup\mathcal{V})
\;\cap\;
(\mathcal{U}\cup\mathcal{W})
\;\cap\;
(\mathcal{V}\cup\mathcal{W}),
\end{equation}
which is simply De Morgan’s law applied to~\eqref{eq:C-def}.  
We verify the two required axioms.
For any $s\in\Sigma$ at least two of the vertices
$\mathcal{U},\mathcal{V},\mathcal{W}$ contain exactly one of$s$
or $s^r$.
A straightforward case analysis using~\eqref{eq:C-Demorgan} shows that
exactly one of $s$ and $s^r$ belongs to $\mathcal{C}$, and so \ref{p1} holds. 
For \ref{p2},  suppose $s\in\mathcal{C}$ and $s\le t\in P$.
Because $s$ lies in at least two of the three given ultrafilters, we may assume
$s\in\mathcal{U}\cap\mathcal{V}$.  
Since $\mathcal{U}$ and $\mathcal{V}$ satisfy Axiom~\ref{p2},
$t$ lies in both, whence $t\in\mathcal{C}$
by~\eqref{eq:C-def}.  Thus $\mathcal{C}$ is closed upward.

Next, we show that $\mathcal{C}$ is the unique common interval vertex.
Let $\mathcal{X}\in I(\mathcal{U},\mathcal{V},\mathcal{W})$.  
We have the following claim:

\begin{clm}\label{cl:interval-bounds}
If $\mathcal{X}\in I(\mathcal{U},\mathcal{V})$ then
$
\mathcal{U}\cap\mathcal{V}\subseteq\mathcal{X}\subseteq\mathcal{U}\cup\mathcal{V}.
$
\end{clm}

\begin{claimproof}
We can compute
\begin{align*}
d(\Ucal,\Vcal)=\frac{1}{2}|\Vcal\triangle \Ucal|=&\frac{1}{2}|(\Vcal\triangle \mathcal{X} )\triangle(\mathcal{X} \triangle\Ucal)|\\
=&\frac{1}{2}|(\Vcal\triangle \mathcal{X} )|+\frac{1}{2}|(\mathcal{X} \triangle\Ucal)|-|(\mathcal{X} \triangle\Ucal)\cap(\mathcal{X} \triangle\Vcal) |\\
=&d(\Vcal,\mathcal X)+d(\Ucal,\mathcal X)-|(\mathcal X\triangle\Ucal)\cap(\mathcal X\triangle\Vcal) |.
\end{align*}
Since $\mathcal{X} \in I(\Ucal,\Vcal)$, we infer that $d(\Ucal,\Vcal)$ is the same as $d(\mathcal{X} ,\Ucal)+d(\mathcal X,\Vcal)$ 
which implies $(\mathcal{X} \triangle\Ucal)\cap(\mathcal{X} \triangle\Vcal) $ must be empty. From this, the claim follows. 
\end{claimproof}

Applying \Cref{cl:interval-bounds} to each pair yields
\[
\mathcal{U}\cap\mathcal{V}\subseteq\mathcal{X}\subseteq\mathcal{U}\cup\mathcal{V},\qquad
\mathcal{U}\cap\mathcal{W}\subseteq\mathcal{X}\subseteq\mathcal{U}\cup\mathcal{W},\qquad
\mathcal{V}\cap\mathcal{W}\subseteq\mathcal{X}\subseteq\mathcal{V}\cup\mathcal{W}.
\]
Intersecting the upper bounds and taking the union of the lower bounds gives
$
\mathcal{C}\subseteq\mathcal{X}\subseteq\mathcal{C}.
$
Hence $\mathcal{X}=\mathcal{C}$, and
$I(\mathcal{U},\mathcal{V},\mathcal{W})=\{\mathcal{C}\}$.
Now, because $\mathcal{X}\in I(\mathcal{U},\mathcal{V})\cap I(\mathcal{U},\mathcal{W})\cap I(\mathcal{V},\mathcal{W})$,
the same distance argument applied to each pair yields
\begin{equation}\label{eq:X-bounds}
(\mathcal{U}\cap\mathcal{V})\cup(\mathcal{V}\cap\mathcal{W})\cup(\mathcal{U}\cap\mathcal{W})
\;\subseteq\;
\mathcal{X}
\;\subseteq\;
(\mathcal{U}\cup\mathcal{V})\cap(\mathcal{V}\cup\mathcal{W})\cap(\mathcal{U}\cup\mathcal{W}).
\end{equation}
However, by \eqref{eq:C-def} and \eqref{eq:C-Demorgan} the lower and upper bounds in
\eqref{eq:X-bounds} coincide and are both equal to $\mathcal{C}$.  We thus have that 
$
I(\mathcal{U},\mathcal{V},\mathcal{W})=\{\mathcal{C}\},
$
and $M_P$ is a median graph. 
\end{proof}

\begin{rem}
    Sageev's construction extends naturally beyond just considering DCC ultrafilters as we have done above, though one needs to consider objects more general than median graphs, namely \emph{median algebras}. This leads to a very general and powerful theory of duality between pocsets and median algebras. We recommend Roller's thesis as a starting point to learn more \cite{roller-thesis}. 
\end{rem}

\subsection{The \texorpdfstring{$\Theta$}{Theta}-classes of the dual median graph}

We now describe how the $\Theta$-classes of the dual median graph $M_P$ are in one-to-one correspondence with elements of $P$, via the labelling function $\ell$.

\begin{prop}
    Let $P$ be a discrete, crossing-finite pocset.  Let $e,f \in \vec E(M_P)$ be such that $e \,\Theta f$. Then $\ell(e) = \ell(f)$. 
\end{prop}

\begin{proof}
    Write $e=(\Ucal_1,\Ucal_2)$,  $f=(\Vcal_1,\Vcal_2)$,  $\ell(e) = s$, and $\ell(f) = t$. 
By construction, we have that
      $\Ucal_2 = \Ucal_1 \triangle \{s,s^r\}$ and  
      $\Vcal_2 = \Vcal_1 \triangle \{t,t^r\}$.
Assume to the contrary that $s \neq t$.
Since $e\,\Theta\,f$, we infer that
$$
d(\Ucal_1,\Vcal_1)+d(\Ucal_2,\Vcal_2)< 
d(\Ucal_1,\Vcal_2)+d(\Ucal_2,\Vcal_1). 
$$
By \Cref{prop:distance-formula}, this implies that
\[
|\Ucal_1\triangle \Vcal_1|+|\Ucal_2\triangle \Vcal_2|
< 
|\Ucal_1\triangle \Vcal_2|+|\Ucal_2\triangle \Vcal_1|.
\]
Set $R\coloneqq\Ucal_1\triangle \Vcal_1$, 
$S\coloneqq\{s,s^r\}$, and  $Q\coloneqq\{t,t^r\}$.
The inequality becomes
$$
\lvert R\rvert
+
\lvert R \triangle S \triangle Q\rvert
< 
\lvert R \triangle Q\rvert
+
\lvert R \triangle S\rvert.
$$
However, we will see that equality actually holds, giving us a contradiction. First, note that
\[
\begin{aligned}
\lvert R\triangle S\rvert
  &= \lvert R\rvert+\lvert S\rvert-2\lvert R\cap S\rvert,\\
\lvert R\triangle Q\rvert
  &= \lvert R\rvert+\lvert Q\rvert-2\lvert R\cap Q\rvert,\\
\lvert R\triangle S\triangle Q\rvert
  &=\lvert R\rvert+\lvert S\rvert+\lvert Q\rvert
     \;-\;2\!\bigl(\lvert R\cap S\rvert+\lvert R\cap Q\rvert+\lvert S\cap Q\rvert\bigr)+\;4\lvert R\cap S\cap Q\rvert.
\end{aligned}
\]
Since $S\cap Q=\varnothing$, we conclude that 
$\lvert S\cap Q\rvert=\lvert R\cap S\cap Q\rvert=0$.
Substitution yields
\[
\begin{aligned}
\lvert R\rvert + \lvert R\triangle S\triangle Q\rvert
  &= \lvert R\rvert + \Bigl(\lvert R\rvert+\lvert S\rvert+\lvert Q\rvert
     -2\lvert R\cap S\rvert-2\lvert R\cap Q\rvert\Bigr)\\
  &= 2\lvert R\rvert + \lvert S\rvert+\lvert Q\rvert
     -2\lvert R\cap S\rvert-2\lvert R\cap Q\rvert\\
  &= \bigl(\lvert R\rvert+\lvert S\rvert-2\lvert R\cap S\rvert\bigr)
     +\bigl(\lvert R\rvert+\lvert Q\rvert-2\lvert R\cap Q\rvert\bigr)\\
  &= \lvert R\triangle S\rvert + \lvert R\triangle Q\rvert.
\end{aligned}
\]   
The result follows.
\end{proof}

\begin{prop}\label{prop:labels-biject-theta-classes}
    Let $P$ be a discrete, crossing-finite pocset.  
    Then the labelling map $\ell : \vec E(M_P) \to P$ descends to a bijection  $\overline \ell : \vec {\mathcal H}(M_P) \to P$. 
\end{prop}

\begin{proof}
    We first show that the induced map $\overline \ell$ is injective. 
    Let $e, f \in \vec E(M_P)$ satisfy $\ell(e) = \ell(f) =:s$. We need only show that $e\,\Theta f$. 
Write $e=(\Ucal_1,\Ucal_2)$, $f=(\Vcal_1,\Vcal_2)$,
Let $Q=\{s,s^r\}$.
We have that $\Ucal_2 = \Ucal_1 \triangle Q$, and $\Vcal_2 = \Vcal_1 \triangle Q$.
Let $R = \Ucal_1 \triangle \Vcal_1$.
By \Cref{prop:distance-formula}, it suffices to show 
$|\Ucal_1\triangle \Vcal_1|+|\Ucal_2\triangle \Vcal_2|
  < 
  |\Ucal_1\triangle \Vcal_2|+|\Ucal_2\triangle \Vcal_1|$.
Substituting $\Ucal_2=\Ucal_1\triangle Q$ and $\Vcal_2=\Vcal_1\triangle Q$, we get
\begin{align*}
  \Ucal_2\triangle \Vcal_2
    &=(\Ucal_1\triangle Q)\triangle(\Vcal_1\triangle Q)
    =\Ucal_1\triangle \Vcal_1 = R,\\
  \Ucal_1\triangle \Vcal_2 &= \Ucal_1\triangle(\Vcal_1\triangle Q)
    =(\Ucal_1\triangle \Vcal_1)\triangle Q = R\triangle Q, \\
  \Ucal_2\triangle \Vcal_1
    &=(\Ucal_1\triangle Q)\triangle \Vcal_1
    =(\Ucal_1\triangle \Vcal_1)\triangle Q = R\triangle Q.
\end{align*}
It follows that we need only prove $|R| < |R\triangle Q|$.
Note that $s\in \Ucal_1\cap \Vcal_1$ which implies that  $s\notin R$.
On the other hand $s^r\notin \Ucal_1\cup \Vcal_1$ which implies that $ s^r\notin R$.
Thus $Q\cap R=\emptyset$.  
Hence
$$
|R\triangle Q|
  \;=\;
  |R|+|Q|-2|R\cap Q|
  \;=\;
  |R|+2-0
  \;=\;
  |R|+2, 
$$
so indeed $|R|<|R\triangle Q|$.  
It follows that
$$
  d(\Ucal_1,\Vcal_1)+d(\Ucal_2,\Vcal_2)
  <
  d(\Ucal_1,\Vcal_2)+d(\Ucal_2,\Vcal_1),
$$
which proves $e\,\Theta\,f$, and so $\overline \ell$ is injective. 

To see that $\overline \ell$ is surjective, by \Cref{lem:flipping-elements} it suffices to show that every $s \in P$ is a minimal element in some DCC ultrafilter. Let $\mathcal U \in V(M_P)$ be arbitrary. Consider $S =\{t \in \mathcal U : t < s\}$. It follows from the fact that $P$ is discrete and crossing-finite, and that $\mathcal U$ is DCC, that $S$ is finite. In particular, it contains a minimal element, which by \Cref{lem:flipping-elements} we may flip to obtain a new ultrafilter with strictly fewer elements smaller than $s$. Repeating this a finite number of times, we eventually uncover a DCC ultrafilter where $s$ is a minimal element. 
\end{proof}

\begin{rem}\label{rem:distinct-labels-on-geodesics}
    Recall that a geodesic in a median graph can cross any given $\Theta$-class at most once; this follows from \Cref{prop:halfspaces-convex}. Combining this with \Cref{prop:labels-biject-theta-classes} above, we deduce that if $P$ is a discrete, crossing-finite pocset, then given any geodesic in the dual median graph $M_P$, the edges in this geodesic are labelled by pairwise distinct elements of $P$, even up to reversal.
\end{rem}


\section{The dual median decomposition}\label{sec:dual-meddecomp}

In this section, we return to the concrete realm of separation systems, and construct a median decomposition over the dual median graph. 

\subsection{Principal ultrafilters}

In order to verify the required axioms when we construct our decomposition, it will be helpful to have access to a certain type of ultrafilter, defined as follows. 

\begin{defn}[Principal ultrafilters]
    Let $G$ be a graph and let $\Sigma$ be a system of separations. Given $x \in V(G)$, an ultrafilter $\mathcal U$ on $\Sigma$ is said to be \defin{$x$-principal} if $x \in A$ for all $(A,B) \in \mathcal U$. 
\end{defn}




%


When the ECC property is satisfied, we gain access to principal DCC ultrafilters. 

\begin{prop}\label{prop:x-principle}
    Let $G$ be a graph and let $\Sigma$ be a discrete, crossing-finite, ECC system of separations. Then for every $x \in V(G)$, there exists a DCC $x$-principal ultrafilter. Moreover, if $(A,B) \in \Sigma$ is such that $x \in A$, then we may take $(A,B) \in \mathcal U$. 
\end{prop}

\begin{proof}
    Let $\Sigma_1$ be the set of all separations $(A,B)\in \Sigma$ such that either $x\in A\sm B$ or $x\in B\sm A$.  
Define
$\Vcal \coloneqq \{\, (A,B)\in \Sigma_1 : x\in  A\sm B \,\}$.
We first show that $\Vcal$ is an ultrafilter on $\Sigma_1$.
By construction, condition \ref{p1} is satisfied: for each separation in $\Sigma_1$, exactly one orientation is chosen.  
Next, suppose $(A,B)\in \Vcal$ and $(C,D)\in \Sigma_1$ with $(A,B)\le (C,D)$. Then $A\subseteq C$, and since $x\in A$, it follows that $x\in C$. Moreover, $D\subseteq B$ and $x\not\in B$, so $x\not\in D$. Hence $x\in C\sm D$, and therefore $(C,D)\in \Vcal$. Thus condition \ref{p2} also holds.
Since $\Sigma$ is ECC, it easily follows that $\Vcal$ is a DCC ultrafilter. We will now extend $\Vcal$ to a DCC ultrafilter on $\Sigma$. 

Let $\Sigma_2 = \Sigma \sm \Sigma_1$. Let us assume without loss of generality that $(A,B) \in \Sigma_2$, lest we already know that $(A,B) \in \Vcal$. We now claim that if $(C,D)\in \Sigma_2$ is nested with some $(A,B)\in \Vcal$, then either $(C,D)\le (A,B)$ or $(D,C)\le (A,B)$.
Suppose, to the contrary, that $(A,B)\le (C,D)$ or $(A,B)\le (D,C)$.  
If $(A,B)\le (C,D)$, then $x\in D\subseteq B$, contradicting $x\in A\sm B$.  
If $(A,B)\le (D,C)$, then $x\in C\subseteq B$, again contradicting $x\in A\sm B$.  
Hence neither case can occur, proving the claim. 
In particular, this implies that if $\mathcal W$ is an ultrafilter on $\Sigma_2$, then $\mathcal V \cup \mathcal W$ is an ultrafilter on $\Sigma$. 
Thus, it suffices for us to find a DCC ultrafilter $\Wcal$ on $\Sigma_2$ such that $(A,B) \in \Wcal$, as we may set $\Ucal := \Vcal \cup \Wcal$ to conclude. 

To see that such a $\mathcal W$ exists, simply consider the median graph $M_{\Sigma_2}$ dual to $\Sigma_2$. This makes sense as $\Sigma_2$ is certainly a discrete and crossing-finite separation system, being a symmetric subset of $\Sigma$. By \Cref{prop:labels-biject-theta-classes}, there exists $e \in \vec E(M_{\Sigma_2})$ labelled by $(A,B)$. Let $\mathcal W \in V(M_{\Sigma_2})$ be the initial vertex of $e$. This is a DCC ultrafilter on $\Sigma_2$ containing $(A,B)$, and thus we are done. 
\end{proof}


\subsection{Constructing the dual decomposition}

We are finally ready to construct a median decomposition over the dual median graph.

\begin{thm}\label{thm:median_decomposition}
Let $G$ be a graph and let $\Sigma$ be  discrete, crossing-finite, ECC separation system.
Define the map 
\[
\beta :
\begin{array}[t]{rcl}
   V\!\bigl(M_\Sigma\bigr) &\longrightarrow& \mathcal{P}\!\bigl(V(G)\bigr)\\[4pt]
   \mathcal{U}               &\longmapsto    & \displaystyle 
      \bigcap_{(A,B)\in\mathcal{U}} A
\end{array}
\]
Then $\bigl(M_\Sigma,\beta\bigr) $ is a median decomposition of $G $.
\end{thm}

\begin{proof}
We verify conditions \ref{m1}–\ref{m3} in turn, beginning with \ref{m1}.

\begin{clm}\label{cl.th.main1.MD1}
We have that
$\displaystyle V(G)=\bigcup_{\mathcal{U}\in V(M_\Sigma)}\beta(\mathcal{U})$.
\end{clm}

\begin{claimproof}
Let $x\in V(G)$ and choose any $x$-principal DCC ultrafilter $\Ucal$, which exists by \Cref{prop:x-principle}. In particular, $\Ucal$ is a vertex of $M_\Sigma$ by construction.
Then $x\in A$ for every $(A,B)\in\Ucal$, hence $x\in \beta(\Ucal)$. Since $x$ was arbitrary, we are done. 
\end{claimproof}

Next, we verify \ref{m2}.

\begin{clm}\label{cl.th.main1.MD2}
For all $uv\in E(G)$, there exists $\mathcal{W}\in V(M_\Sigma)$
such that $u,v\in \beta({\mathcal{W}})$.
\end{clm}

\begin{claimproof}
Let $uv\in E(G)$. By \Cref{cl.th.main1.MD1}, let $\Ucal \in V(M_\Sigma)$ be such that $u \in \beta(\Ucal)$. 
If $v \in A$ for all $\ab \in \Ucal$, then we are done.
Otherwise, define
\[
   Y
   \coloneqq\{(A,B)\in\mathcal{U} :  v\in B\sm A\},
   \qquad
   \mathcal{W}\coloneqq(\mathcal{U}\sm Y)\cup Y^{r},
\]
where $Y^{r}\coloneqq\{(B,A) : (A,B)\in Y\}$.
In what follows we prove that $\mathcal{W}$ is a vertex of $M_\Sigma$ and that 
$u,v\in \beta(\mathcal{W})$. First, we must verify that $\mathcal W$ is  an ultrafilter. 
Indeed, it is clear from the construction of $\mathcal W$ that \ref{p1} is satisfied, as for every $(A,B)\in\Sigma$, exactly one of $(A,B)$ or $(B,A)$ lies in $\mathcal{W}$.

Next we verify \ref{p2}, i.e. that $\mathcal W$ is closed upwards.
Assume for a contradiction that there exists $(A,B) \in \mathcal W$, $(C,D) \in \Sigma$ with $(A,B) < (C,D)$, but$(C,D)\notin\mathcal{W}$.  
By \ref{p1}, this implies $(D,C)\in\mathcal{W}$.  
We now distinguish two main cases.

Suppose first that $(D,C)\in Y^{r}$.
If $(A,B)\in Y^{r} $,  then $(B,A)\in Y$ and $(C,D)\in Y$, so $v\in A\sm B$ and $v\in D\sm C$.  
Because $(A,B)<(C,D)$ and $v\in D$, we would have $v\in B$, contradicting $v\in A\sm B$.
If instead $(A,B)\in\mathcal{U}\sm Y$, then since $(A,B)\in\mathcal{U}\sm Y$, we have $v\notin B\sm A$.  
From $\ab\le \cd $ it follows that $v\in C$.  
But $(D,C)\in Y^{r} $ implies $(C,D)\in Y$, and so $v\in D\sm C$, another contradiction.
Suppose instead that $(D,C)\in\mathcal{U}\sm Y$.
If $(A,B)\in Y^{r} $, then because $(D,C)\in\mathcal{U}\sm Y$, we have $v\notin C\sm D$.  
Since $\ab\le (C,D)$, we deduce $C\subseteq A$ and thus $v\in C$ which yields a contradiction.
If instead $(A,B)\in\mathcal{U}\sm Y$, then since $\mathcal{U}$ contains $(A,B)$ and $(A,B)\le \cd$, we conclude that $(C,D)\in\mathcal{U}$. But $(D,C)\in\mathcal{U}$ violates \ref{p1} since both orientations cannot coexist.
In every case above we reached a contradiction; hence $(C,D)\in\mathcal{W}$.
It follows that $\mathcal{W}$ is indeed an ultrafilter. 

It remains to check that for every $(A,B) \in \mathcal{W}$, 
we have $u,v \in A$. To this end, let
 $(A,B)\in \mathcal{W}$.
We distinguish the following two cases. First, suppose that $(A,B)\in \mathcal{U}\sm Y$. In this case, since $(A,B)\in \Ucal$, we have that $u\in A$.
Also $\ab\notin Y$ which implies that $v\in A$.
Thus we have $u,v \in A$.
Suppose instead that $(A,B)\in Y^{r}$. In this case, since $(B,A)\in Y$, we have that $v\in A\sm B$.
If $\ab\in \Ucal$, then $u\in A$, as $u\in \beta(\Ucal)$.
So we can assume that $\ba \in \Ucal $.
Since $u\in \beta(\Ucal)$, we infer that $u\in B$.
Given $u\in B$ and $v\in A\sm B$, we must have $u\in A\cap B$ to preserve the edge $uv$ in $G$.
Thus, we once again conclude that $u,v\in A$.

It follows that $\mathcal W$ is both $u$-principal and $v$-principal. Moreover, it is clear that $\mathcal W$ is DCC, since it has finite symmetric difference with $\mathcal U$. This implies that $\mathcal W$ is indeed a vertex of $M_\Sigma$ and $u,v \in \beta(\mathcal W)$. Whence the claim.
\end{claimproof}

Finally, we check that \ref{m3} holds.

\begin{clm}\label{cl.th.main1.MD3}
For all $v\in V(G)$, the fibre
$
   \beta^{-1}(v)
$
induces a convex subgraph of $M$.
\end{clm}

\begin{claimproof}
Assume without loss of generality that $|\beta^{-1}(v)|\ge2$.  
Let $\mathcal{U},\mathcal{V}\in \beta^{-1}(v)$ and let
$\gamma\coloneqq\mathcal{U}=\mathcal{U}_0,\dots,\mathcal{U}_k=\mathcal{V}$ be a geodesic in $M_\Sigma$.  
We will show $V(\gamma)\subseteq \beta^{-1}(v)$.
Suppose $\mathcal{U}_i\notin \beta^{-1}(v)$, where $i$ here is the minimal index where this happens.
Then $(A,B)\in\mathcal{U}_i$ for some $(A,B)$ with $v\notin A$.
Because labels along a geodesic are pairwise distinct (see \Cref{rem:distinct-labels-on-geodesics}), the same ordered separation
appears in $\mathcal{V}=\mathcal{U}_k$, contradicting $v\in \beta(\mathcal{V})$.
Hence no such $\mathcal{U}_i$ exists and so $\beta^{-1}(v)$ induces a convex subgraph of $M$.
\end{claimproof}

With~\Cref{cl.th.main1.MD1}–\Cref{cl.th.main1.MD3} established,
$(M_\Sigma,\beta)$ satisfies \ref{m1}–\ref{m3} and is therefore a
median–decomposition of~$G$.
\end{proof}

We now give a name to the object defined above, which is the main topic of interest in this paper. 

\begin{defn}[The dual median decomposition]
    Let $G$ be a graph and $\Sigma$ be a discrete, crossing-finite, ECC separation system. Let $(M_\Sigma, \beta)$ be the median decomposition constructed in \Cref{thm:median_decomposition}. Then $(M_\Sigma, \beta)$ is called the \defin{median decomposition dual to $\Sigma$}. 
\end{defn}

\begin{rem}
    Note that if $G $ is a $\Gamma$-graph for some group $\Gamma$, and $\Sigma$ is a $\Gamma$-invariant discrete symmetric separation system. 
    Then it is easy to verify that the dual median decomposition $(M_\Sigma, \beta)$ is $\Gamma$-canonical.
\end{rem}

\subsection{Duality and uniqueness}\label{sec:duality}

We now establish some key properties of the dual decomposition. 
Namely, we formulate and prove a version of Sageev--Roller duality for median decompositions. This is achieved in two halves.

\begin{lem}
    Let $G$ be a graph and $\Sigma$ a discrete, crossing-finite, ECC system of separations. Let $(M_\Sigma, \beta)$ be the dual median decomposition. Let $\overline \ell : \vec {\mathcal H}(M_\Sigma) \to \Sigma$ be the bijection described in \Cref{prop:labels-biject-theta-classes}. 
    Let $(A,B) \in \Sigma$, and $\mathfrak h \in \vec {\mathcal H}(M_\Sigma)$ satisfy $ \overline \ell(\mathfrak h) = (A,B)$. Then $(A_{\mathfrak h}, B_{\mathfrak h}) = (A,B)$. 
\end{lem}

\begin{proof}
    We first show that $A_{\mathfrak h}\subseteq A$.
    For any $\mathcal{U}\in \mathfrak h_+$, we have $(A,B)\in\mathcal{U}$.  Since $\beta(\mathcal{U})
    =\bigcap_{\cd\in\mathcal{U}}C$,
    and $(A,B)\in\mathcal{U}$, it follows that $\beta(\mathcal{U})\subseteq A$.  Thus,
    \[
    A_{\mathfrak h}
    =\bigcup_{\mathcal{U}\in \mathfrak h_+}\beta(\mathcal{U})
    \subseteq A.
    \]
    We now show that $A \subset A_\mathfrak h$. Let $x \in A$. By \Cref{prop:x-principle}, there exists an $x$-principal DCC ultrafilter $\mathcal U$ on $\Sigma$, such that $(A,B) \in \mathcal U$. Note that $\mathcal U \in \mathfrak h_+$, and $\beta(\mathcal U) \ni x$. This implies that $x \in A_\mathfrak h$. Since $x \in A$ was arbitrary, we have shown the claimed inclusion.
    We therefore conclude that $A_\mathfrak h = A$ by double inclusion. An identical argument applied to $\mathfrak h^r$ gives  $B_\mathfrak h = B$, and we are done.
\end{proof}

Importantly, the above result tells us that the labelling function we defined on the edges of the dual median graph via minimal elements of ultrafilters in \Cref{sec:sageev} agrees with the labelling induced by the median decomposition and its dual separation system introduced in \Cref{sec:med-decomp}. This result can be stated more cleanly as follows.

\begin{thm}
    Let $G$ be a graph and $\Sigma$ a discrete, crossing-finite, ECC system of separations. Let $(M_\Sigma, \beta)$ be the dual median decomposition. Then $(M_\Sigma, \beta)$ is reduced. The set of separations dual to $(M_\Sigma,\beta)$ is exactly $\Sigma$. 
\end{thm}

In the other direction, duality is not quite as straightforward for median decompositions. For one, we saw in \Cref{exa:dual-seps-not-ECC} that the dual separation system to a median decomposition might not be discrete or ECC, and so we may not be able to form the dual median decomposition again. Even if we are given this, the other problem is that there will generally be many reduced median decompositions labelled by a given system of separations. However, there is a strong relationship between these median decompositions and the true dual, described by the following theorem.

\begin{thm}\label{prop:median-decomp-contains-the-dual}
    Let $G$ be a graph and let $\Sigma$ be a discrete, crossing-finite, ECC separation system. Let $(M,\alpha)$ be a reduced median decomposition such that its dual system of separations satisfies  $\Sigma_{M,\alpha} = \Sigma$. Let $(M_\Sigma, \beta)$ be the median decomposition dual to $\Sigma$. Then there is a median embedding $\varphi : M_\Sigma \to M$ such that $\beta \circ \varphi = \alpha$. 
    
    Moreover, suppose further that $(M,\alpha)$ is crossing-faithful. Then $\varphi$ is surjective and thus a graph isomorphism. 
\end{thm}

\begin{proof}
    Since $(M,\alpha)$ is reduced, we have natural bijection $g : \vec {\mathcal H}(M) \to \Sigma$, which is necessarily monotone and reversal-preserving (though not necessarily an isomorphism of pocsets). Define $\varphi$ as follows. Given a DCC ultrafilter $\mathcal U \in V(M_\Sigma)$, consider $g^{-1}(\mathcal U)$. This is a DCC ultrafilter on $\vec {\mathcal H}(M)$, and thus the intersection
    $$
    I = \bigcap_{\mathfrak h \in g^{-1}(\mathcal U)} \mathfrak h_+
    $$
    contains exactly one element, say $I = \{v\}$, as remarked in \Cref{rem:ultrafilters-on-hyperplanes}. We therefore set $\varphi(\mathcal U) := v$. Recall from the proof of \Cref{lem.it.is.median.graph} that the unique median of three DCC ultrafilters $\mathcal U$, $\mathcal V$, $\mathcal W \in V(M_\Sigma)$ is given by
$$
(\mathcal{U}\cap\mathcal{V})
\;\cup\;
(\mathcal{U}\cap\mathcal{W})
\;\cup\;
(\mathcal{V}\cap\mathcal{W}).
$$
In particular, since $g^{-1}$ clearly respects intersections and unions, we have that $\varphi$ preserves medians. Since $g$ is surjective, we have that $\varphi$ is injective \cite[Prop.~7.8(ii)]{roller-thesis}. Moreover, it is also clear that adjacency is preserved, since two ultrafilters define point to adjacent vertices exactly when their symmetric difference has size 2.

    The fact that $\beta \circ \varphi = \alpha$ follows immediately from the construction of $(M_\Sigma, \beta)$ and \Cref{prop:bags-are-intersections}. 

    Finally, suppose that $(M,\alpha)$ is crossing-faithful. Then the map $g$ is an isomorphism of pocsets, and so any DCC ultrafilter on $\vec {\mathcal H}(M)$ is labelled by a DCC ultrafilter on $\Sigma$. In particular, since we have a natural bijection between DCC ultrafilters on $\vec {\mathcal H}(M)$ and $V(M)$, this means that $\varphi$ is surjective.
\end{proof}

It follows from \Cref{prop:median-decomp-contains-the-dual} that the dual median decomposition is always crossing-faithful. 
Even better, we actually deduce the following strong uniqueness statement.

\begin{cor}[Uniqueness of the dual median decomposition]\label{cor:uniqueness}
    Let $G$ be a connected graph and let $\Sigma$ be a discrete, crossing-finite, ECC system of separations. Then the dual median decomposition $(M_\Sigma, \beta)$ is the \textbf{unique} reduced, crossing-faithful median decomposition with dual system of separations equal to $\Sigma$.
\end{cor}

\begin{exa}
    We return to the reduced median decomposition $(M,\beta)$ of the path graph of length 3, $G$, seen in \Cref{exa:not-crossing-faithful}. The dual system of separations $\Sigma$ to this median decomposition is given by 
    \begin{align*}
    \Sigma = \Big\{ \ &\big( \{1,2\},\{2,3,4\}\big), \big(\{2,3,4\},\{1,2\}\big), \\
    &\big(\{1,2,3\},\{3,4\}\big), \big(\{3,4\},\{1,2,3\}\big) \ \Big\}.
    \end{align*}
    This is nested, and so the dual median decomposition  is in fact a tree decomposition $(T_\Sigma, \alpha)$. This is depicted in \Cref{fig:path-exa-of-subdecomp}, along with the embedding of this tree decomposition inside of $(M,\beta)$.

  \begin{figure}[h]
        \centering
        \subfloat[\centering $(T_{\Sigma},\alpha)$]{{\includegraphics[scale=0.7]{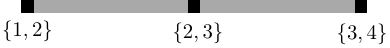}}}
        \qquad
        \subfloat[\centering $(M,\beta)$.]{{\includegraphics[scale=0.7]{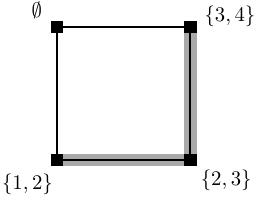} }}%
        \caption{}
        \label{fig:path-exa-of-subdecomp}
    \end{figure}
\end{exa}


\subsection{Examples}

We now run through some instructive examples.

\begin{exa}[Structure trees]
As we remarked in the introduction, it is a classical fact, due to Dunwoody, that a \emph{nested} discrete system of separations gives rise to a \defin{structure tree}; see \cite{connectivity,MR4316013}.  
We now verify that, in the case where the separation system is nested, the dual median graph construction recovers the classical structure tree. This is a well-established phenomenon, and so we only include a sketch of the details for demonstrative purposes. 

For completeness, we briefly recall how to construct a tree from a nested set of separations.
Fix a nested, discrete, ECC separation system $\Sigma$ of $G$.
Define an equivalence relation~$\sim$ on~$\Sigma$ by
\[
(A,B)\sim(C,D)\;:\Longleftrightarrow\;
\begin{cases}
(A,B)=(C,D),\\[2pt]
\text{or $(B,A)$ is a predecessor of $(C,D)$ in $(\Sigma,\le)$}. 
\end{cases}
\]
(The order $\le$ is the usual partial order on separations; a
\defin{predecessor} is a maximal element below $(C,D)$.)
For $(A,B)\in\Sigma$ write $[(A,B)]$ for its equivalence class, and let
$T_{\Sigma}$ be the graph whose vertices are these classes.  
Join $[(A,B)]$ to $[(B,A)]$ with an oriented edge labelled by 
$(A,B)$. Of course, the reverse orientation is labelled by $(B,A)$. 
We now show that there is a natural label-preserving graph isomorphism between $M_\Sigma$ and $T_{\Sigma}$.


First, we claim that, given $\mathcal U \in V(M_\Sigma)$, 
if $\ab, \cd \in \Ucal$ are both minimal elements then $[\ab]=[\cd]$. 
Indeed, since $\Sigma$ is nested, $\ab$ is comparable with either $\cd$ or $\dc$.  
Hence, one of the following cases must occur:
\[
\dc \le \ab, \quad \cd \le \ab, \quad \ab \le \cd, \quad \text{or} \quad \ab \le \dc.
\]
Because $\ab$ and $\cd$ are both minimal in $\Ucal$, it follows that either $\dc \le \ab$ or $\ab \le \dc$.  
Assume that $\ab \le \dc$.  
By~\ref{p1}, this implies that $\dc \in \Ucal$, contradicting the assumption that $\cd \in \Ucal$.  
Therefore, we must have $\dc \le \ab$.
Now assume that there exists a separation $(E,F) \in \Sigma$ such that $\dc \le (E,F) \le \ab$.  
Since $\ab$ is minimal in $\Ucal$ and $(E,F) \le \ab$, we deduce that $(E,F) \notin \Ucal$, and consequently, $(F,E) \in \Ucal$.  
Moreover, because $(D,C) \le (E,F)$, we have $(F,E) \le (C,D)$.  
This contradicts the assumption that $\cd$ is minimal in $\Ucal$.  
Hence, $\dc$ is a predecessor of $\ab$, which implies that $[\ab] = [\cd]$, as desired.




We can now define the map $\Phi\colon V(M_\Sigma) \to V(T_{\Sigma})$. 
Let $\Ucal $ be a vertex of $M_\Sigma$.
Consider the minimal elements in $\Ucal$. 
We know by the above that all of them belong to the same class of equivalence relation under $\sim$.
This defines a unique vertex of $T_{\Sigma}$, which we designate as $\Phi(\mathcal U)$. 
Let $\Ucal, \Vcal \in V(M_\Sigma)$ be adjacent. Then there exists a minimal separation $s \in \Ucal$ such that  
$\Vcal = (\Ucal \sm \{s\}) \cup \{s^r\}$, labelling the edge from $\mathcal U$ to $\mathcal V$. 
By definition,  
$\Phi(\Ucal) = [s]$  and $ \Phi(\Vcal) = [s^r]$.
By the construction of $T_\Sigma$, there is an edge between $[s]$ and $[s^r]$ labelled by $s$.  Hence, $\Phi$ preserves edges and their labels, and so in particular is a morphism of graphs. 
It is easy to see that $\Phi$ maps neighbourhoods of vertices injectively. Since $T_\Sigma$ is a tree, it follows that $\Phi$ is injective. 

Finally, we check that $\Phi$ is surjective. In particular, we claim that given  $\ab\in \Sigma$, there is a vertex $\Ucal$ such that $\ab$ is minimal in $\Ucal$. 
By \Cref{prop:x-principle}, there exists an $x$-principal DCC ultrafilter $\Ucal$ such that $(A,B) \in \Ucal$.  
Note that there are only finitely many separations $(C,D) \in \Ucal$ with $(C,D) \leq (A,B)$, since $\Sigma$ is discrete.
Let  $\mathcal{A} = \{(C,D) \in \Ucal  :  (C,D) \leq (A,B)\}$.
Take a minimal element $(C,D)$ of $\mathcal{A}$. Then $(C,D)$ is also minimal in $\Ucal$.  
By \Cref{lem:flipping-elements}, the set  
$\Ucal' = \Ucal \sm \{(C,D), (D,C)\}$  
is an ultrafilter.  
Since $\mathcal{A}$ is finite, after finitely many steps we end up with an ultrafilter in which $(A,B)$ is a minimal element of $\Ucal$.  
This proves the claim.
This easily implies that $\Phi$ is surjective.
Indeed, take an arbitrary vertex $[\ab]$ of $T_{\Sigma}$.
By the above, there exists ultrafilter $\Ucal$ such that $\Ucal$ has $\ab$ as a minimal element.
One can see that $\Phi(\Ucal)=[\ab]$.

\end{exa}


\begin{exa} 
Consider the following graph $G$, with vertex set $V(G) = \{1,2,3,4,5\}$ depicted in \Cref{fig:exa1}. 
Let 
\begin{align*}
&s_1 = ( \{1,2,3\},\{1,3,4,5\}),  &s_2 = ( \{1,2,4\},\{2,3,4,5\}), \\ &s_3 = ( \{1,2,3,4\},\{3,4,5\}).
\end{align*}
Consider the separation system
    $\Sigma = \{ s_1,s_2,s_3, s_1^r,s_2^r,s_3^r, \}$.
We have that $s_1$ and $s_2$ cross, but $s_3$ is nested with respect to both $s_1$ and $s_2$. The median decomposition dual to $\Sigma$ is also depicted in \Cref{fig:exa1}.
  \begin{figure}[h]
        \centering
        \subfloat[\centering The graph $G$.]{{\includegraphics[scale=0.7]{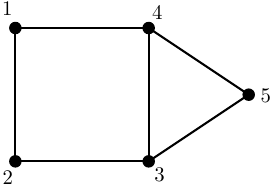}}}
        \qquad
        \subfloat[\centering The median decomposition dual to $\Sigma$.]{{\includegraphics[scale=0.6]{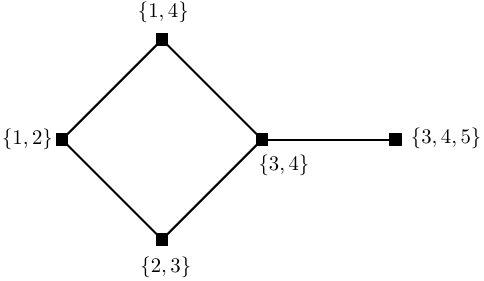} }}%
        \caption{}
    \label{fig:exa1}
\end{figure}
\end{exa}

\begin{exa}\label{exa:inf_grid}
The next example is about the infinite square grid $G$.
Let $\ab$ be the separation obtained by a vertical double ray, see \Cref{fig:exa-median_2} (a).
Then let $\Sigma$ be the orbit of $\ab$ under the automorphism group of $G$. This is easily seen to be a discrete, crossing-finite, ECC system of separations. The median graph $M_\Sigma$ dual to $\Sigma$ is also an infinite square grid. 
\Cref{fig:exa-median_2} (b) indicates the corresponding dual median decomposition.
\begin{figure}[h]
\centering
\subfloat[\centering The separation $\ab$ of the grid]{{\includegraphics[scale=0.8]{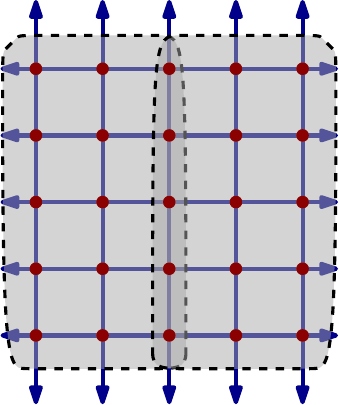}}}
\qquad
\subfloat[\centering The dual median graph with a bag of the dual of median decomposition $(M_\Sigma, \beta)$.]{{\includegraphics[scale=0.8]{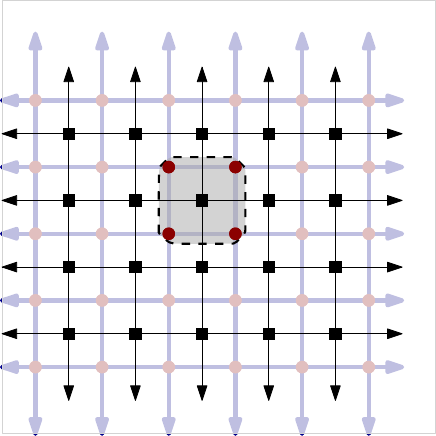} }}%
\caption{}\label{fig:exa-median_2}

\end{figure}

\end{exa}

\begin{exa}
let $G$ be the triangle augmented cuboctahedral graph. Consider the separation $(A,B)$ depicted in \Cref{fig:exa-cuboc}(a). Let $\Sigma$ denote the orbit of $(A,B)$ under the automorphism group of $G$. This is a separation system consisting of exactly 6 separations (and thus 3 up to orientation reversal). Ignoring reversals, these separations all pairwise cross. The dual median graph $M_\Sigma$ is a cube, and the dual median decomposition is also depicted in \Cref{fig:exa-cuboc}(b).

\begin{figure}[h]
\centering
\subfloat[\centering The separation $(A,B)$.]{{\includegraphics[scale=0.6]{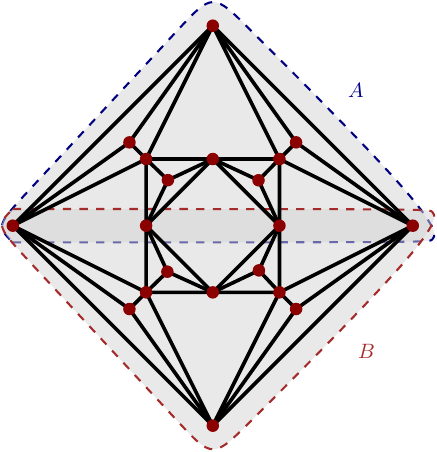}}}
\qquad
\subfloat[\centering The dual median graph $M_\Sigma$.]{{\includegraphics[scale=0.8]{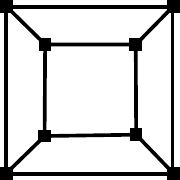} }}%
\quad
\subfloat[\centering The dual median decomposition $(M_\Sigma,\beta)$.]{{\includegraphics[scale=0.6]{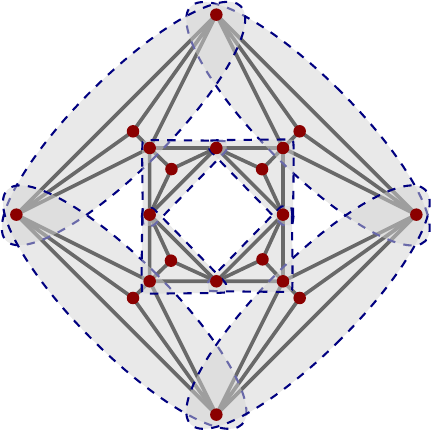}}}
    \label{fig:exa-cuboc}
    \caption{The triangle-augmented cuboctahedral graph.}
\end{figure}
\end{exa}

\section{Stavropoulos' characterisation of clique number}\label{sec:clique}

In this section apply our machinery for constructing median decompositions, and extend a theorem of Stavropoulos. In \cite{stavropoulos2015medianwidth}, Stavropoulos shows that the median-width of a finite graph is exactly its clique number. The goal of this section is to give a short proof of this fact using the machinery constructed above, which extends to \emph{all} graphs. 

\begin{prop}\label{prop:cliques-in-bags}
    Let $G$ be a graph and $(M,\beta)$ a median decomposition. Then $S = \{v_1, \ldots, v_n\} \subset V(G)$ span a finite clique in $G$. Then there exists $u \in V(M)$ such that $S \subset \beta(u)$.
\end{prop}

\begin{proof}
    Let $F_i = \beta^{-1}(v_i)$. By \ref{m3}, each is a convex subset of $M$. By \ref{m2}, $F_i \cap F_j \neq \emptyset$ for all $i, j$. By \Cref{thm:helly}, we have that $\bigcap_i F_i \neq \emptyset$. Let $u$ lie in this intersection, then $S \subset \beta(u)$. 
\end{proof}

Compare the above with the analogous statement for tree decompositions \cite[Cor.~12.3.5]{diestel}.

\begin{thm}\label{thm:cliquebags}
Let $G$ be a graph which contains no infinite clique.
Then there exists a median-decomposition $(M,\beta)$ of $G$ such that every bag is a finite clique.
\end{thm}

\begin{proof}
We introduce some notation. Give $u \in V(G)$, 
The \defin{star separation} defined by $u$ is $(A_u, B_u) := (N[u],V(G)\sm \{u\})$.
We denote the symmetric set of all star separations with their flips by $\Sigma_{\mathsf{star}}$, i.e. 
$$
\Sigma_{\mathsf{star}}=\bigl\{\, (A_u, B_u),\;(B_u, A_u)  :  u\in V(G) , N[u] \neq V(G)\bigr\}.
$$
Observe that we do not include separations associated to vertices which neighbour every other vertex in the graph. This is to filter out non-essential separations.
It is then clear that $\Sigma:= \Sigma_{\mathsf{star}}$ is a discrete ECC separation system. Moreover, since $G$ does not contain an infinite clique, we have that $\Sigma_{\mathsf{star}}$ is crossing-finite. Indeed, let $u,v \in V(G)$ and suppose that the separations $(A_u,B_u)$ and $(A_v, B_v)$ cross. Then it is easy to see that $u$ and $v$ must be adjacent. In particular, if $\Sigma$ is not crossing-finite then $G$ contains an infinite clique. 

    Consider the dual median decomposition $(M_\Sigma,\beta)$. We claim that every bag in this decomposition is a clique. Let $u,v \in V(G)$ be not adjacent. Then $A_v \not\ni u$ and $B_v \not\ni v$. This implies that $(A_v, B_v), (A_u, B_u) \in \Sigma$. In particular, for any DCC ultrafilter $\mathcal U$ on $\Sigma$, we have that 
    $$
    \{u,v\} \not\subset \bigcap_{(A,B) \in \mathcal U} A = \beta(\mathcal U).
    $$
    Thus, vertices in the same bag are pairwise adjacent, and thus bags span cliques.
\end{proof}



The following corollary is now immediate.

\begin{cor}[cf. {\cite[Theorem 5.12]{stavropoulos2016graph}}]\label{cor:clique-med-decomp}
Let $G$ be a graph. 
Then $\mathsf{mw}(G)=\omega(G)$.
\end{cor}


\begin{exa}
Let $G$ be the cycle of length 5, with vertices $V(G) = \{1,2,3,4,5\}$.
We have 
$$
\Sigma_{\mathrm{star}}=\{\,s_1,s_2,s_3,s_4,s_5,\ s_1^{r},s_2^{r},s_3^{r},s_4^{r},s_5^{r}\,\},
$$
where
\begin{align*}
    &s_1=(\{1,2,5\},\{2,3,4,5\}),
&s_2=(\{1,2,3\},\{1,3,4,5\}),\\
&s_3=(\{2,3,4\},\{1,2,4,5\}),
&s_4=(\{3,4,5\},\{1,2,3,5\}),\\
&s_5=(\{1,4,5\},\{1,2,3,4\}).
\end{align*}
The dual median decomposition to this system of separations is depicted in \Cref{fig:exa-median}.
\begin{figure}[H]
\includegraphics[scale=0.7]{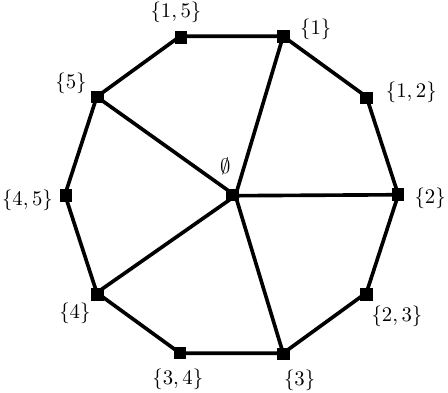} 

\caption{}\label{fig:exa-median}
\end{figure}

\end{exa}

We now take a retrospective look at the original methods of Stavropoulos used to prove \Cref{cor:clique-med-decomp} for finite graphs, and compare them to our own.
Stavropoulos made use of the well-known fact that Cartesian products of median graphs are themselves median, with easily described convex subgraphs. 
%
%
Using this, Stavropoulos described how one can take a finite collection of tree decompositions of a graph $G$, and construct the `Cartesian product' of these decompositions. More precisely, we have the following. 

\begin{prop}[{cf. \cite[Theorem 6.7]{stavropoulos2016graph}}]
\label{lem:prod-tree-decomp-median-decomp}
Let $G$ be a graph and for each $i\in[n]$ let $(M_i,\beta_i)$ be a median decomposition of $G$.
Set $M:=M_1\cartprod\cdots\cartprod M_n$ and define
\[
\beta:V(M)\longrightarrow \mathcal P\bigl(V(G)\bigr),\qquad
\beta(t_1,\dots,t_n):=\bigcap_{i=1}^n \beta_i(t_i).
\]
Then $(M,\beta)$ is a median decomposition of $G$.
\end{prop}

\begin{proof}
    It is a standard fact that a Cartesian product of median graphs is a median graph. The convex subgraphs of $M$ are precisely those of the form $C=C_1\cartprod\cdots\cartprod C_k$, where each $C_i$ is a convex subgraph of $M_i$ \cite[Lem.~5.1]{stavropoulos2016graph}.

    Given $x \in V(G)$, let $v_i \in V(M_i)$ be such that $x \in \beta_i(v_i)$. Then $x \in \beta(v_1, \ldots, v_n)$. This implies that $(M,\beta)$ satisfies \ref{m1} and \ref{m2}. To see that \ref{m3} holds, we note that $\beta^{-1}(x) = \beta_1^{-1}(x) \cartprod\cdots\cartprod\beta_n^{-1}(x)$, and so fibres are convex by the above remark.
\end{proof}

Stavropoulos only considered the above in the case where the $M_i$ are trees, and so the $(M_i,\beta_i)$ are tree decompositions. However, the argument goes through very much unchanged. 
One curious thing about this construction is that it provides an easy way to construct `exotic' median decompositions. For example, median decompositions whose dual system of separations is not discrete, as we saw in \Cref{exa:dual-seps-not-ECC}. More precisely, let $\Sigma$ be a separation system of a graph $G$ which decomposes as a finite disjoint union 
$$
\Sigma = \Sigma_1 \sqcup \cdots \sqcup \Sigma_k,
$$
where each $\Sigma_i$ is itself a discrete, crossing-finite, and ECC separation system. Then \Cref{lem:prod-tree-decomp-median-decomp} allows us to easily construct a median decomposition $(M,\beta)$ of $G$ such that $\Sigma_{M,\beta} = \Sigma$ by taking $(M_i, \beta_i)$ to be the median decomposition dual to $\Sigma_i$, even if $\Sigma$ itself is not discrete. Of course, $\Sigma$ is necessarily ECC and crossing-finite if the $\Sigma_i$ are. 
If $\Sigma$ is discrete and the true dual median decomposition $(M_\Sigma, \alpha)$ is defined, then this product decomposition $(M,\beta)$ will usually strictly subsume the dual $(M_\Sigma, \alpha)$, containing it as a `sub-decomposition' in the sense of \Cref{prop:median-decomp-contains-the-dual}. We obtain equality between these decompositions if and only if for all $i \neq j$, every element of $\Sigma_i$ crosses every element of $\Sigma_j$.


\section{Geometry of median decompositions}\label{sec:geometry}

The results of the previous section highlight something important. 
That is, that the median-width parameter encodes zero information about the global geometry of a graph. 
This is in stark contrast to tree-width and tree-length, which are now well-understood to strongly influence the large-scale structure of graphs; see \cite{hickingbotham2025graphs, nguyen2025coarse, distel2025alternative}.  
That being said, this does not mean that median decompositions do not themselves encode geometry. In this section, we will take a particular look at special types of median decompositions---so-called \emph{proper} and \emph{geometric} decompositions---and describe how their existence influences the coarse geometry of the graph. Special attention will be paid to Cayley graphs of groups acting on median graphs.

\subsection{Pull-back decompositions}\label{sec:pull-backs}

First, we describe a cheap but very helpful trick for constructing a median decomposition of a graph over a given median graph.

\begin{defn}[Pull-back median decomposition]\label{def:pull-back}
    Let $G$ be a connected graph, and let $M$ be a median graph. Let $\varphi : V(G) \to V(M)$ be any set-theoretic map. For every $\mathfrak h \in \vec {\mathcal H}(M)$, let 
    $$
    (A_\mathfrak h, B_\mathfrak h) := (N[\varphi^{-1}(\mathfrak h_+)], N[\varphi^{-1}(\mathfrak h_-)]). 
    $$
    Define a map $\beta : V(M) \to \mathcal P(V(G))$ via
    $$
    \beta(v) := \bigcap_{\substack{\mathfrak h \in \vec {\mathcal H}(M) : \\ v \in \mathfrak h_+}} A_\mathfrak h.
    $$
    Then $(M,\beta)$ is called the \defin{pull-back median decomposition over $\varphi$}. 
\end{defn}

\begin{prop}
    Let $G$ be a graph, $M$ a median graph, and $\varphi : V(G) \to V(M)$. Then the pull-back median decomposition $(M,\beta)$ is a median decomposition.
\end{prop}

\begin{proof}
    We verify the axioms in turn. First, note that if $xy \in E(G)$, then both $x$ and $y$ lie in $\beta(\varphi(x))$. This verifies both \ref{m1} and \ref{m2}. 
    For \ref{m3}, let $x \in V(G)$ and let $F = \beta^{-1}(x)$ denote the set of bags containing $x$. Let $u,v \in F$, and let $\gamma$ be a geodesic between $u$ and $v$. Suppose there existed $w \in V(\gamma)$ such that $w \not\in F$. Then there exists $\mathfrak h \in \vec {\mathcal H}(M)$ such that $w \in \mathfrak h_+$, but $x \not\in A_\mathfrak h$. But then we must have that $u,v \in \mathfrak h_-$. This contradicts the assumption that $\gamma$ is a geodesic $u$ and $v$, since it must crosses some $\Theta$-class twice. 
\end{proof}

\begin{rem}
    We suggest the reader take particular note of the simple fact that $x \in \beta(\varphi(x))$ for all $x \in V(G)$ as it will prove to be quite helpful.
\end{rem}

\begin{defn}
    Suppose that $G$ and $H$ are $\Gamma$-graphs, then a map $\varphi : V(G) \to V(H)$ is said to be \defin{$\Gamma$-equivariant} if $g \cdot \varphi(x) = \varphi(g\cdot x)$ for all $x \in V(G)$, $g \in \Gamma$.
\end{defn}

\begin{rem}
    If $G$ and $M$ are $\Gamma$-graphs in \Cref{def:pull-back} and $\varphi : V(G) \to V(M)$ is $\Gamma$-equivariant, then is easy to see that the pull-back median decomposition over $\varphi$ is $\Gamma$-canonical.
\end{rem}

\subsection{Proper and geometric median decompositions}

We now study some special types of median decompositions. We will need the following definitions. 

\begin{defn}\label{def:proper-and-geometric}
    Let $G$ be a graph, and let $(M, \beta)$ be a median decomposition. 
    \begin{enumerate}
        \item We say that $(M, \beta)$ has \defin{bounded fibres} there exists $R \geq 0$ such that for all $x \in V(G)$, we have $\diam(\beta^{-1}(x)) \leq R$.  

        \item Call $(M,\beta)$ \defin{properly supported} if for all $r \geq 0$ there exists $R \geq 0$ such that for all $v \in V(M)$, we have
        $$
        \diam\Big(\{x \in V(G) : d(v,\beta^{-1}(x)) \leq r \}\Big) \leq R. 
        $$ 

        \item Finally, $(M, \beta)$ is \defin{quasi-densely supported} if there exists $R \geq 0$ such that for all $v \in V(M)$ there exists $u \in V(M)$ such that $d(u,v) \leq R$ and $\beta(u) \neq \emptyset$. 
    \end{enumerate}
    If $(M, \beta)$ is properly supported and has bounded fibres, we call it \defin{proper}. If, in addition, $(M,\beta)$ is also quasi-densely supported, we say that $(M,\beta)$ is \defin{geometric}. 
\end{defn}

We now describe how proper and geometric median decompositions influence the geometry of a graph.
Recall that a map $f : X \to Y$ between graphs a \defin{coarse embedding} if there exists $K_1 \geq 1$ and an increasing function $\rho : \N \to \N$ such that $\rho(n) \to \infty$ as $n \to \infty$, and 
    $$
    K_1d(x,y) + K_1\geq d(f(x),f(y)) \geq \rho(d(x,y))
    $$
for all $x,y \in X$. If $\rho$ is affine, i.e. there exists $K_2 \geq 1$ such that $\rho(n) = \tfrac 1 {K_2} n- K_2$, then $f$ is called a \defin{quasi-isometric embedding}. 
If there exists $K_3 \geq 0$ such that for all $y \in Y$ there exists $x \in X$ such that $d(f(x),y) \leq K_3$, then $f$ is said to be \defin{coarsely surjective}. A coarsely surjective quasi-isometric embedding is called a \defin{quasi-isometry}, and $X$ and $Y$ are said to be \defin{quasi-isometric}. Note that coarse embeddings and quasi-isometries between graphs need not be graph morphisms. Moreover, it is known that a coarsely surjective coarse embedding between two connected graphs is necessarily a quasi-isometry \cite[Cor.~4.3]{de2016zooming}. 

To give the correct statement of the next theorem, we need to define the `cubical subdivision' of a median graph. The following definition is fairly standard in the literature on median graphs and CAT(0) cube complexes; see e.g. \cite{haglund2023isometries}.

\begin{defn}[First cubical subdivision]
    Let $M$ be a median graph, and $n \geq 0$. An \defin{$n$-cube} in $M$ is an induced subgraph isomorphic to the $n$-dimensional cube graph. Note that 0-cubes correspond to vertices, and 1-cubes correspond to edges. Define a new graph $M'$ as follows:
    \begin{enumerate}
        \item $V(M) = \{C \subset M : \text{$C$ is a cube}\}$.
    
        \item Two cubes are joined by an edge if one is contained in  the other, and their dimensions differ by exactly 1. 
    \end{enumerate}
    The graph $M'$ is called the \defin{first cubical subdivision} of $M$. 
\end{defn}

It is easy to show that the first cubical subdivision $M'$ of a median graph $M$ is also a median graph. If $M$ is a $\Gamma$-graph, then this action extends naturally to $M'$. The natural inclusion $V(M) \hookrightarrow V(M')$ has the effect of doubling the distance between vertices in the image. If $M$ is crossing-bounded, then this inclusion is coarsely surjective and thus a quasi-isometry. If $C \subset M$ is a convex subgraph, then the first cubical subdivision $C'$ of $C$ naturally includes into $M'$ as a convex subgraph. Note that the first cubical subdivision of a bounded median graph is also necessarily bounded. 

It is a standard fact that if a group $\Gamma$ acts on a median graph $M$ such that the orbit $\Gamma \cdot v$ of some/every vertex is bounded in diameter, then $\Gamma$ stabilises a cube and thus admits a global fixed-point in $M'$, see  \cite{gerasimov1998fixed}. 



We are now ready to state and prove the next theorem, which offers a translation between (canonical) proper/geometric median decompositions and (equivariant) coarse embeddings/quasi-isometries. 

\begin{thm}\label{thm:proper-geom-char}
    Let $\Gamma$ be a group, $G$ be a connected $\Gamma$-graph, and $M$ be a median $\Gamma$-graph.
    \begin{enumerate}
        \item Suppose $G$ admits a $\Gamma$-canonical median decomposition $(M,\beta)$ over $M$. 
        \begin{enumerate}
            \item If $(M,\beta)$ is proper, then there exists a $\Gamma$-equivariant coarse embedding  $f : G \to M'$, where $M'$ is the first cubical subdivision of $M$.

            \item If $M$ is crossing-bounded and $(M,\beta)$ is geometric, then there exists a $\Gamma$-equivariant quasi-isometry  $f : G \to M'$.
        \end{enumerate} 


        
        \item Suppose that either $G$ is bounded degree, or $M$ is crossing-bounded. Let $f : G \to M$ be a $\Gamma$-equivariant coarse embedding (quasi-isometry). Then the pull-back median decomposition over $f$ is $\Gamma$-canonical and proper (geometric). 
    \end{enumerate}
\end{thm}

The hypotheses in \Cref{thm:proper-geom-char}(1) are a little convoluted, though this is necessary to ensure equivariance of the map $f$. Mercifully, these hypotheses can often be weakened in specific settings to yield simpler statements. This is discussed below in \Cref{rem:discussion-of-proper-geom-char-hyp}.

\begin{proof}[Proof of \Cref{thm:proper-geom-char}]
    First, we prove (1)(a). Suppose that $G$ admits a $\Gamma$-canonical proper median decomposition $(M, \beta)$. 
    We claim that there exists a $\Gamma$-equivariant coarse embedding of $G$ into $M'$, the first cubical subdivision of $M$. Since $\Gamma$ acts on $G$, let $\{x_\alpha\}_\alpha$ be orbit representatives. Let $\Gamma_\alpha \leq \Gamma$ denote the stabiliser of $x_\alpha$. 
    We have that $\Gamma_\alpha$ acts on the convex subgraph $F_\alpha:= M[\beta^{-1}(x_\alpha)]$.
    Since $(M,\beta)$ has bounded fibres, the $F_\alpha$ are uniformly bounded in diameter. In particular, this means that $\Gamma_\alpha$ fixes a vertex in $F_\alpha'$, and thus in $M'$, say $v_\alpha \in V(M')$, and $v$ lies uniformly close to $\beta^{-1}(x_\alpha) \subset V(M) \subset V(M')$. Extend the map $x_\alpha \mapsto v_\alpha$ equivariantly to a $\Gamma$-equivariant map $f : V(G) \to V(M')$. Since $(M,\beta)$ has bounded fibres, by \ref{m2} it is easy to see that there exists $K_1 > 0$ such that $K_1d(x,y) + K_1\geq d(f(x),f(y))$ for all $x,y \in V(G)$. The required lower bound then follows easily from the definition of properly supported, since we must have that $d(\beta^{-1}(x), \beta^{-1}(y)) \to \infty$ uniformly as $d(x,y) \to \infty$. It follows that $f$ is a coarse embedding.
    
    To see (1)(b), suppose further that  $(M, \beta)$ is quasi-densely supported and $M$ is crossing-bounded. Since the inclusion $V(M) \hookrightarrow V(M')$ is coarsely surjective and $(M,\beta)$ has bounded fibres, we must have that $f$ is coarsely surjective. Since coarsely surjective coarse embeddings of graphs are quasi-isometries \cite[Cor.~4.3]{de2016zooming}, it follows that $f$ is a quasi-isometry. This concludes the proof of (1). 
    
    We now prove (2). Let $f : V(G) \to V(M)$  be a $\Gamma$-equivariant coarse embedding. Let $(M, \beta)$ be the $\Gamma$-canonical pull-back median decomposition over $f$, as defined in \Cref{def:pull-back}. We claim that $(M,\beta)$ is proper. First, we show that it has bounded fibres. Fix $x \in V(G)$.
    Since fibres are convex by \ref{m3}, let $e_1e_2\ldots e_n$ be a geodesic of length $n$ in $\beta^{-1}(x)$, with $n$ large. Let $\mathfrak h^i$ denote the oriented $\Theta$-class of $e_i$. In particular, since this path was a geodesic, the $\mathfrak h^i$ are all pairwise distinct. The construction of the pull-back decomposition gives us separations via
    $$
    (A_i, B_i) := (A_{\mathfrak h^i}, B_{\mathfrak h^i}) = (N[f^{-1}(\mathfrak h^i_+)], N[f^{-1}(\mathfrak h^i_-)]).
    $$
    Since each of the $e_i$ has both endpoints in $\beta^{-1}(x)$, we must have that
    $
    x \in A_i \cap B_i
    $. 
    We must therefore have that for every $i \geq 1$ there exists $y_i, z_i \in N[x]$ such that $y_i \in f^{-1}(\mathfrak h^i_+)$, and $z_i \in f^{-1}(\mathfrak h^i_-)$. We now split into two cases.
    
    Suppose first that $M$ is crossing-bounded. By passing to a long subsequence, $\mathfrak h^{n_1}, \ldots, \mathfrak h^{n_k}$, we may assume that the $\mathfrak h^i$ are nested, where $k = k(n)$ is an increasing function in $n$. But then $f(y_{n_1})$ is separated from $f(z_{n_k})$ by every $\mathfrak h^i$. In particular, $d(f(y_{n_1}), f(z_{n_k})) \geq k$. But $d(y_{n_1}, z_{n_k}) \leq 2$ since every $y_i, z_i \in N[x]$. This bounds $k$, thus bounding $n$. It follows that $(M,\beta)$ has bounded fibres. 
    
    Suppose instead that $G$ is bounded-degree. Then, again, by passing to a long subsequence $\mathfrak h^{n_1}, \ldots, \mathfrak h^{n_k}$, we may assume that $y_{n_i} = y_{n_j} =: y$ and $z_{n_j} = z_{n_i} =: z$ for all $1 \leq i,j \leq k$, where, again, $k = k(n)$ is an increasing function in $n$. But then $f(y)$ and $f(z)$ are separated by $k$ $\Theta$-classes, and so $d(f(y),f(z)) \geq k$. However, $d(y,z) \leq 2$, and so we once again bound $k$ and thus bound $n$. In both cases, we have shown that $(M, \beta)$ has bounded fibres. 

    Recall that $x \in \beta(f(x))$ for all $x \in V(G)$. In particular, since $(M,\beta)$ has bounded fibres and $f$ is a coarse embedding, this easily implies that $(M,\beta)$ is properly supported, and thus a proper median decomposition.
    Assume further that $f$ is coarsely surjective. Then, again, since $x \in \beta(f(x))$ for all $x \in V(G)$, this immediately implies that $(M, \beta)$ has quasi-dense support, and so thus is a geometric median decomposition. This concludes the proof of (2).
\end{proof}

\begin{rem}\label{rem:discussion-of-proper-geom-char-hyp}
    Some remarks are in order:
    \begin{enumerate}
        \item For full generality, passing to the cubical subdivision is necessary to ensure equivariance of $f$ in \Cref{thm:proper-geom-char}(1). However, it is not always required. 
        In particular, if for all $x \in V(G)$ we have that the stabiliser $\Gamma_x \leq \Gamma$ fixes a vertex in $M$, then the argument goes through without passing to the cubical subdivision. 
        In this case, we can also drop the crossing-bounded hypothesis of \Cref{thm:proper-geom-char}(1)(b), as it is only needed to ensure coarse surjectivity of the inclusion into the cubical subdivision.

        \item One recovers a simpler statement for graphs without a group action by simply setting $\Gamma = \{1\}$. As noted above, this case does not require passing to the cubical subdivision for \Cref{thm:proper-geom-char}(1) to hold.

        \item For (non-equivariant) quasi-isometries, a related result appears for \emph{graph decompositions} in Knappe's thesis \cite[\S4.3]{knappe2025graph}.
    \end{enumerate}
    
\end{rem}

The following example illustrates that \Cref{thm:proper-geom-char}(2) fails without the additional hypothesis that either $M$ is crossing-bounded or $G$ is bounded-degree.

\begin{exa}
    Let $G$ be an infinite star graph. Let $M$ denote the infinite cube. That is, the standard Cayley graph of the group $\Gamma = \bigoplus_{\mathbb N} \mathbb Z_2$. There is an isometric embedding $f : G \to M$ which sends the central vertex of $G$ to the identity vertex of $M$, and each leaf of $G$ to a standard basis vector. The pull-back decomposition over $f$ has all fibres unbounded. In particular, it is not proper in the sense of \Cref{def:proper-and-geometric}. 
\end{exa}


As a corollary of the above, we obtain the following characterisation of when a finitely generated group $\Gamma$ acts metrically-properly or geometrically on a median graph $M$.  
Recall that a group action by $\Gamma$ on $M$ is said to be \defin{metrically proper} if the set 
$$
\{g \in G : d(x_0, g\cdot x_0) \leq R\}
$$
is finite for all $R \geq 0$. We say that the action is \defin{cobounded} if there exists $R \geq 0$ such that for all $x,y \in M$, there exists $g \in \Gamma$ such that $d(y, g \cdot x) \leq R$. An action which is both metrically-proper and cobounded is called \defin{geometric}.

\begin{cor}\label{cor:geom-med-decomp-cayley-graphs}
    Let $\Gamma$ be a finitely generated group and $S \subset \Gamma$ a finite generating set. Let $G = \cay(\Gamma,S)$ be the Cayley graph. Let $M$ be a median graph. Then $\Gamma$ acts metrically-properly (geometrically) on $M$ if and only if $G$ admits a $\Gamma$-canonical proper (geometric) median decomposition $(M,\beta)$. 
\end{cor}

The study of finitely generated groups acting metrically-properly or geometrically on median graphs is a rich source of interesting mathematics. \Cref{cor:geom-med-decomp-cayley-graphs} provides a characterisation of this phenomenon purely in terms of the structural properties of the Cayley graph which, in particular, is more widely applicable to $\Gamma$-graphs in general. This leaves us with a dictionary, which opens the door to think about whether there are results in geometric group theory relating to group actions on median graphs admitting non-trivial generalisations in structural graph theory.


\subsection{Proper decompositions from separation systems}

We conclude this section by describing necessary and sufficient conditions for the median decomposition dual to a separation system to be proper. These are described in the following definition. 

\begin{defn}
    Let $G$ be a graph, and let $\Sigma$ be system of separations. 
    \begin{enumerate}
        \item We say that $\Sigma$ has \defin{uniform local width} if
            there exists $R \geq 0$ such that for all $x \in V(G)$, every chain and anti-chain in the subset
            $
            \Sigma_x := \{(A,B) \in \Sigma : x \in A \cap B\}
            $
            has cardinality at most $R$. 

        \item We say that $\Sigma$ has \defin{increasing separations} if for all $r \geq 0$ there exists $N = N(r) \geq 0$, where $N(r) \to \infty$ as $r \to \infty$, such that for all $x,y  \in V(G)$ with $d(x,y) \geq r$, there exists at least $N$ pairwise distinct $(A,B) \in \Sigma$ such that $x \in A \sm B$ and $y \in B \sm A$.
    \end{enumerate}
\end{defn}

\begin{thm}\label{thm:proper-dual-decomp-char}
    Let $G$ be a graph.
    \begin{enumerate}
        \item Let $\Sigma$ be a discrete, crossing-finite, ECC system of separations. Let $(M_\Sigma, \beta)$ denote the dual median decomposition.
        \begin{enumerate}
            \item If $\Sigma$ has uniform local width then $(M_\Sigma, \beta)$ has bounded fibres.

            \item If $\Sigma$ has increasing separations then $(M_\Sigma, \beta)$ is properly supported.
        \end{enumerate}
        \item Conversely, let $(M,\beta)$ be a weakly reduced median decomposition, with $M$ crossing-finite. Let $\Sigma = \Sigma_{M,\beta}$ denote the dual system of separations.
        \begin{enumerate}
            \item If $(M, \beta)$ has bounded fibres then $\Sigma$ has uniform local width.

            \item If $(M, \beta)$ is properly supported then $\Sigma$ has increasing separations.
        \end{enumerate}
\end{enumerate}
\end{thm}

\begin{proof}
    We first prove (1). 
    Suppose $\Sigma$ has uniform local width. 
    Fix $x \in V(G)$ and let $e_1, \ldots, e_n$ be a long geodesic in the fibre $\beta^{-1}(x)$. Let $\mathfrak h^i \in \vec {\mathcal H}(M_\Sigma)$ denote the $\Theta$-class of $e_i$ and $(A_i,B_i) \in \Sigma$ be the associated separation. Note that each $(A_i,B_i)$ is distinct. By Ramsey's theorem, there exists either a chain or anti-chain of size $k=k(n)$ in $\{(A_i,B_i)\}_i$, where $k$ is a strictly increasing function in $n$. In particular, since the size of chains and anti-chains $\Sigma_x$ is bounded we have that $n$ is necessarily bounded, and so $(M_\Sigma, \beta)$ has bounded fibres. This concludes the proof of (1)(a).

    Suppose now that $(M_\Sigma,\beta)$ is not properly supported. This means that there exists $r \geq 0$ such that for all $R \geq 0$, there exists $x, y \in V(G)$ such that $d(x,y) \geq r$ and $d(\beta^{-1}(x), \beta^{-1}(y)) \leq R$. In particular, by \Cref{thm:kakutani}, at most $R$ distinct unoriented $\Theta$-classes separate $\beta^{-1}(x)$ from $\beta^{-1}(y)$. By \Cref{lem:fibre-halfspace-char}, this means that $\Sigma$ contains at most $R$ elements $(A,B)$ such that $x \in A \sm B$ and $y \in B \sm A$. This proves (1)(b).

    We now prove (2).
    Note first that $\Sigma$ is necessarily a discrete, crossing-finite, ECC separation system by \Cref{prop:nest-cross-induced-seps} and \Cref{prop:weakly-reduced-crossing-finite-gives-nice-seps}.
    Suppose that $\Sigma_x := \{(A,B) \in \Sigma : x \in A \cap B\}$ contains a chain or anti-chain $(A_1,B_1), \ldots (A_n, B_n)$ of length $n$. Let $\mathfrak h^i$ be an oriented $\Theta$-class labelled by $(A_i,B_i)$. Since $(M,\beta)$ is weakly reduced, by \Cref{prop:nest-cross-induced-seps} we have that $\mathfrak h^i_\pm \cap \mathfrak h^j_\pm \neq \emptyset$ for all $i, j$. Then by Lemma~\ref{lem:fibre-halfspace-char}, we have that $\beta^{-1}(x)$ intersects $\mathfrak h^i_+$ and $\mathfrak h^i_-$ for all $i$. Since fibres are convex by \ref{m3}, and halfspaces are convex by \Cref{prop:halfspaces-convex}, by \Cref{thm:helly} we have that there exists 
    $$
    u \in \beta^{-1}(x) \cap \bigcap_{i} \mathfrak h^i_+ \ \ \ \text{and} \  \  \ v \in \beta^{-1}(x) \cap \bigcap_{i} \mathfrak h^i_-.
    $$  
    This implies that $d(u,v) \geq n$ by \Cref{thm:kakutani}, and so the fibre of $x$ in $(M, \beta)$ has diameter at least $n$. This concludes the proof of (2)(a). 

    Suppose now that $(M, \beta)$ is properly supported. This implies that for all $r \geq 0$ there exists $R \geq 0$ such that for all $x,y \in V(G)$ with $d(x,y) \geq R$, we have that $d(\beta^{-1}(x), \beta^{-1}(y)) \geq r$. Again, by \Cref{thm:kakutani} and \Cref{lem:fibre-halfspace-char}, we have that $\Sigma$ has increasing separations. This proves (2)(b).
\end{proof}


\begin{cor}\label{cor:coarse-embed-trichotomy}
    Let $G$ be a $\Gamma$-graph. Then the following are equivalent. 
    \begin{enumerate}
        \item\label{thm:trichotomy1} There exists a $\Gamma$-equivariant coarse embedding of $G$ into a crossing-bounded median graph. 

        \item\label{thm:trichotomy2} There exists a $\Gamma$-canonical, reduced, proper median decomposition of $G$ over a crossing-bounded median graph.

        \item\label{thm:trichotomy3} $G$ admits a $\Gamma$-invariant discrete, crossing-bounded, ECC separation system with uniform local width and increasing separations.
    \end{enumerate}
\end{cor}

\begin{proof}
    We have that (\ref{thm:trichotomy3}) $\implies$ (\ref{thm:trichotomy2}) by \Cref{thm:proper-dual-decomp-char}(1). Then, (\ref{thm:trichotomy2}) $\implies$ (\ref{thm:trichotomy1}) follows from \Cref{thm:proper-geom-char}(1)(a). 
    Finally, to see that (\ref{thm:trichotomy1}) $\implies$ (\ref{thm:trichotomy3}), suppose that there exists a $\Gamma$-equivariant coarse embedding $f : G \to M$ of $G$ into a crossing-bounded median graph. By \Cref{thm:proper-geom-char}(2), the pull-back median decomposition $(M,\beta)$ over $f$ is $\Gamma$-canonical and proper. However, $(M,\beta)$ may not be weakly reduced. Let $(N,\alpha)$ denote the weak reduction of $(M,\beta)$ described in \Cref{rem:weak-reduction}. It is easy to see that $(N,\alpha)$ is also proper, but now weakly reduced. Also, $N$ is crossing-bounded since $M$ is.
    By \Cref{prop:nest-cross-induced-seps}, \Cref{prop:weakly-reduced-crossing-finite-gives-nice-seps}, and \Cref{thm:proper-dual-decomp-char}(2), 
    we deduce that (\ref{thm:trichotomy1}) $\implies$ (\ref{thm:trichotomy3}). 
\end{proof}

As an application, we remark that the above results can be used to prove bounds on coarse invariants such as asymptotic dimension. For example, we have the following corollary.

\begin{cor}
    Let $G$ be a graph. If $G$ admits a discrete, crossing-bounded, ECC system of separations with uniform local width and increasing separations, then $G$ has finite asymptotic dimension, bounded above by the maximal cardinality of a set of pairwise-crossing elements of $\Sigma$. 
\end{cor}

\begin{proof}
    The asymptotic dimension of a median graph is bounded above by the maximal dimension of its cubes \cite{wright2012finite}, and asymptotic dimension is monotone under coarse embeddings. The maximal dimension of a cube in $M_\Sigma$ is exactly the maximal cardinality of any set of pairwise-crossing elements of $\Sigma$. The claim now follows from \Cref{thm:proper-dual-decomp-char} and \Cref{thm:proper-geom-char}.
\end{proof}





\begin{rem}\label{rem:quasi-dense}
    It is likely not possible to provide a natural description of quasi-dense support for median decompositions in terms of their dual separation systems without additional hypotheses. This problem is related to the fact that cobounded-ness of actions on dual median graphs in geometric group theory is also somewhat elusive. Indeed, there exists fairly straightforward examples of groups acting cofinitely on a pocset where the induced action on the dual median graph is metrically proper but not cobounded \cite{chatterji2005wall}. As such, finding sufficient conditions for the action on the dual median graph to be cobounded is generally a topic of interest; for example, a very useful theorem of Sageev provides such a criterion in the context of hyperbolic groups \cite[\S3]{sageev1997codimension}; see also \cite[Thm.~2.6]{hsu2010cubulating} for a more self-contained statement.
\end{rem}

\bibliographystyle{plainurlnat}
\bibliography{gc}

\end{document}